\documentclass[12pt,a4paper]{article}
\usepackage{hyperref}
\usepackage[biblatex,cleveref]{anziamjedraft}
\usepackage{amsfonts,stmaryrd}
\bibliography{bibexport}
\title{Implementation of high-order,\\ 
discontinuous Galerkin time stepping\\
for fractional diffusion problems}
\author{William McLean}
\address{School of Mathematics and Statistics,
University of New South Wales, Sydney, NSW~2052, \textsc{Australia}.}
\mailto{w.mclean@unsw.edu.au}
\myorcid{0000-0002-7133-2884}

\date{\today}
\newcommand{\jump}[1]{\llbracket#1\rrbracket}
\newcommand{\iprod}[1]{\langle#1\rangle}
\newcommand{\bigiprod}[1]{\bigl\langle#1\bigr\rangle}

\begin{document}
\maketitle
\begin{abstract}
The discontinuous Galerkin (\textsc{dG}) method provides a robust and flexible
technique for the time integration of fractional diffusion problems.  
However, a practical implementation uses coefficients defined by integrals
that are not easily evaluated.  We describe specialised 
quadrature techniques that efficiently maintain the overall accuracy of the 
\textsc{dG} method.  In addition, we observe in numerical experiments that
known superconvergence properties of \textsc{dG} time stepping for classical
diffusion problems carry over in a modified form to the fractional-order
setting.
\end{abstract}
\tableofcontents
\section{Introduction}\label{sec: intro}
The discontinuous Galerkin (\textsc{dG}) method provides an effective
numerical procedure for the time integration of diffusion problems.  In the
mid-1980s, Eriksson, Johnson and Thom\'ee~\cite{ErikssonEtAl1985} provided
the first detailed error analysis, which has been subsequently extended and
refined by numerous 
authors~\cite[][and references therein]{MakridakisNochetto2006,%
SchmutzWihler2019,SchotzauSchwab2001}.  The \textsc{dG} method has also
proved effective for time stepping of \emph{fractional} diffusion 
problems~\cite{McLeanMustapha2009,Mustapha2015} of the form
\begin{equation}\label{eq: fpde}
\partial_tu+\partial_t^{1-\alpha}Au=f(t)
	\quad\text{for $0<t\le T$, with $u(0)=u_0$.}
\end{equation}
Here, $A$ is a linear, second-order, elliptic partial differential operator
over a spatial domain~$\Omega$, subject to a homogeneous Dirichlet boundary
condition~$u=0$ on~$\partial\Omega$.  (Our notation suppresses the dependence 
of $u$~and $f$ on the spatial variables.) The fractional diffusion exponent is 
assumed to satisfy~$0<\alpha<1$ (the \emph{sub}-diffusive case), and the 
fractional time derivative is understood in the Riemann--Liouville sense: 
for $t>0$~and $\mu>0$,
\[
\partial_t^\mu v=\frac{\partial}{\partial t}\int_0^t\omega_\mu(t-s)v(s)\,ds
\quad\text{where}\quad \omega_\mu(t)=\frac{t^{\mu-1}}{\Gamma(\mu)}.
\]

The partial integro-differential equation~\eqref{eq: fpde} arises in a 
variety of physical models~\cite{KlafterSokolov2011,MetzlerKlafter2000} of 
diffusing particles whose behaviour is described by a continuous-time random 
walk for which the waiting-time distribution is a power law that decays 
like~$1/t^{1+\alpha}$.  The expected waiting time is therefore infinite, and 
the mean-square displacement is proportional to~$t^\alpha$.
Standard Brownian motion is recovered in the limit as~$\alpha\to1$,
when \eqref{eq: fpde} reduces to the classical diffusion equation.

Our main concern in the present work is with the practical implementation of
\textsc{dG} time stepping for~\eqref{eq: fpde}, and in particular with the
accurate evaluation of certain coefficients~$H^{n,n-\ell}_{ij}$ used
during the $n$th step.  \cref{sec: ODE} introduces the \textsc{dG}
method for the fractional ODE case of~\eqref{eq: fpde}, in which the 
operator~$A$ is replaced by a scalar~$\lambda>0$.  We will see in the 
simplest, lowest-order scheme, when the \textsc{dG} solution is 
piecewise-constant in time, that
\[
H^{n,0}_{11}=\int_{t_{n-1}}^{t_n}\frac{d}{dt}\biggl(
	\int_{t_{n-1}}^t\omega_\alpha(t-s)\,ds\biggr)\,dt
\]
and
\[
H^{n,n-\ell}_{11}=\int_{t_{n-1}}^{t_n}\frac{d}{dt}\biggl(
	\int_{t_{\ell-1}}^{t_\ell}\omega_\alpha(t-s)\,ds\biggr)\,dt
\quad\text{for $1\le\ell\le n-1$,}
\]
where $0=t_0<t_1<t_2<\cdots$ are the discrete time levels.
We easily verify that $H^{n,0}_{11}=\omega_{\alpha+1}(k_n)
=k_n^\alpha/\Gamma(\alpha+1)$, for a step-size~$k_n=t_n-t_{n-1}$, and
\begin{equation}\label{eq: H lowest order}
\begin{aligned}
H^{n,n-\ell}_{11}&=\omega_{\alpha+1}(t_n-t_{\ell-1})
	-\omega_{\alpha+1}(t_n-t_\ell)\\
	&\qquad{}-\omega_{\alpha+1}(t_{n-1}-t_{\ell-1})
	+\omega_{\alpha+1}(t_{n-1}-t_\ell),
\end{aligned}
\end{equation}
but for higher-order schemes the coefficients become progressively
more complicated. Although the~$H^{n,n-\ell}_{ij}$ can always be 
evaluated via repeated integration by parts, the resulting expressions are 
likely to suffer from roundoff when evaluated in floating-point arithmetic
if $n-\ell$ is large.  Consider just the lowest order 
case~\eqref{eq: H lowest order} with uniform time steps~$t_n=nk$, so that
\[
H^{n,n-\ell}_{11}=k^\alpha\bigl[\omega_{\alpha+1}(n-\ell+1)
	-2\omega_{\alpha+1}(n-\ell)+\omega_{\alpha+1}(n-\ell-1)\bigr].
\]
Since the factor in square brackets is a second-difference 
of~$\omega_{\alpha+1}$, its magnitude decays like $(n-\ell)^{\alpha-2}$ 
as $n-\ell$ increases, but the individual terms grow like $(n-\ell)^\alpha$.

We are therefore led to evaluate the coefficients~$H^{n,n-\ell}_{ij}$ via
quadratures with positive weights.  No special techniques are needed for 
$\ell\le n-2$, but when $\ell=n$~or $n-1$ we must deal with weakly singular 
integrands. In \cref{sec: evaluation}, we show how certain substitutions reduce 
the problem to dealing with integrands that are either smooth, or are products 
of smooth functions and standard Jacobi weight functions.  Similar 
substitutions, known as Duffy transformations~\cite{Duffy1982}, have long been 
used to compute singular integrals arising in the boundary element method.

\cref{sec: spatial} introduces a spatial discretisation for the fractional
PDE~\eqref{eq: fpde} and describes the structure of the linear system
that must be solved at each time step.  In \cref{sec: Legendre}, we
specialise the expressions for the coefficients by choosing Legendre
polynomials as the shape functions employed in the \textsc{dG} time stepping.

\cref{sec: reconstruction} describes a post-processing technique that,
when applied to the \textsc{dG} solution~$U$, produces a more accurate
approximate solution~$\widehat U$, known as the 
reconstruction~\cite{MakridakisNochetto2006} of~$U$.  If
$U$ is a piecewise polynomial of degree at most~$r-1$, then 
$\widehat U$ is a piecewise polynomial of degree at most~$r$.  For a 
classical diffusion problem, both $U$~and $\widehat U$ are known to be 
quasi-optimal, that is, accurate of order $k^r$~and $k^{r+1}$, respectively.
Thus, it is natural to ask what happens in the fractional-order
case, and we investigate this question in numerical experiments 
reported in \cref{sec: numerical}.
\section{A fractional \textsc{ODE}}\label{sec: ODE}
Our central concern is present already in the zero-dimensional case when we 
replace the elliptic operator~$A$ with a scalar~$\lambda\ge0$, so that the 
solution~$u(t)$ is a real-valued function satisfying the fractional \textsc{ODE}
\begin{equation}\label{eq: scalar ivp}
u'+\lambda\partial_t^{1-\alpha}u=f(t)
	\quad\text{for $0<t\le T$, with $u(0)=u_0$.}
\end{equation}
For the time discretisation, we introduce a grid
\[
0=t_0<t_1<t_2<\cdots<t_N=T,
\]
and form the vector~$\boldsymbol{t}=(t_0,t_1,\ldots,t_N)$.
Let $k_n=t_n-t_{n-1}$ denote the length of the $n$th (open) 
subinterval~$I_n=(t_{n-1},t_n)$.  We form the disjoint union
\[
I=I_1\cup I_2\cup\cdots\cup I_N,
\]
and for any function~$v:I\to\mathbb{R}$ write
\[
v^n_+=\lim_{\epsilon\downarrow 0}v(t+\epsilon),\qquad
v^n_-=\lim_{\epsilon\downarrow 0}v(t-\epsilon),\qquad
\jump{v}^n=v^n_+-v^n_-,
\]
provided the one-sided limits exist.

Given a vector $\boldsymbol{r}=(r_1,r_2,\ldots,r_N)$ of integers~$r_n\ge0$,
the trial space~$\mathcal{X}=\mathcal{X}(\boldsymbol{t},\boldsymbol{r})$ 
consists of the functions $X:I\to\mathbb{R}$ such that 
$X|_{I_n}\in\mathbb{P}_{r_n-1}$ for~$1\le n\le N$.  Here, $\mathbb{P}_m$ 
denotes the space of polynomials of degree at most~$m\ge0$, with real 
coefficients.  The \textsc{dG} solution~$U\in\mathcal{X}$ 
of~\eqref{eq: scalar ivp} is then defined by 
\cite{McLeanMustapha2009,Mustapha2015}
\begin{equation}\label{eq: scalar dG}
\jump{U}^{n-1}X^{n-1}_++\int_{I_n}(U'+\lambda\partial_t^{1-\alpha}U)X\,dt
	=\int_{I_n}fX\,dt
\end{equation}
for $X\in\mathbb{P}_{r_n-1}$ and $1\le n\le N$, where, in the case~$n=1$, 
we set $U^0_-=u_0$ so that $\jump{U}^0=U^0_+-U^0_-=U^0_+-u_0$.  
(The monograph of Thom\'ee~\cite[Chapter~12]{Thomee2006} is a standard 
reference providing a general introduction to \textsc{dG} time stepping for 
classical diffusion problems.)

To compute~$U$, we choose for each~$n$ a basis $\psi_{n1}$, $\psi_{n2}$,
\dots, $\psi_{nr_n}$ for~$\mathbb{P}_{r_n-1}$ and write
\begin{equation}\label{eq: U psi}
U(t)=\sum_{j=1}^{r_n}U^{nj}\psi_{nj}(t)\quad\text{for $t\in I_n$.}
\end{equation}
When $X=\psi_{ni}$, we find that
\[
U^{n-1}_+X^{n-1}_++\int_{I_n}U'X\,dt=\sum_{j=1}^{r_n}G^n_{ij}U^{nj}
\quad\text{and}\quad
U^{n-1}_-X^{n-1}_+=\sum_{j=1}^{r_{n-1}}K^{n,n-1}_{ij}U^{n-1,j},
\]
with coefficients given by
\[
G^n_{ij}=\psi_{nj}(t_{n-1})\psi_{ni}(t_{n-1})
	+\int_{I_n}\psi_{nj}'\psi_{ni}\,dt
\]
and
\[
K^{n,n-1}_{ij}=\psi_{n-1,j}(t_{n-1})\psi_{ni}(t_{n-1}).
\]
Owing to the convolutional structure of the fractional derivative, it is 
convenient to introduce the notation
\[
\bar\ell=n-\ell
\]
and define, if $t\in I_n$,
\[
\rho^{n\bar\ell}_j(t)=\rho^{n,n-\ell}_j(t)
	=\int_{I_\ell}\omega_\alpha(t-s)\psi_{\ell j}(s)\,ds
\quad\text{for $1\le\ell\le n-1$,}
\]
with
\[
\rho^{n\bar n}_j(t)=\rho^{n0}_j(t)
	=\int_{t_{n-1}}^t\omega_\alpha(t-s)\psi_{nj}(s)\,ds.
\]
We find that
\[
\partial_t^{1-\alpha}U=\sum_{\ell=1}^n\sum_{j=1}^{r_\ell} 
    U^{\ell j}(\rho^{n\bar\ell}_j)'(t) \quad\text{for $t\in I_n$,}
\]
and thus
\[
\int_{I_n}(\partial_t^{1-\alpha}U)X\,dt
	=\sum_{\ell=1}^n\sum_{j=1}^{r_\ell}H^{n\bar\ell}_{ij}U^{\ell j}
\quad\text{where}\quad H^{n\bar\ell}_{ij}=H^{n,n-\ell}_{ij}
	=\int_{I_n}(\rho^{n\ell}_j)'\psi_{ni}\,dt.
\]
Hence, putting
\[
F^{ni}=\int_{I_n}f\psi_{ni}\,dt,
\]
the \textsc{dG} method~\eqref{eq: scalar dG} requires 
\begin{multline}\label{eq: basic linear system}
\sum_{j=1}^{r_n}\bigl(G^n_{ij}+\lambda H^{n0}_{ij}\bigr)U^{nj}=F^{ni}
-\sum_{\ell=1}^{n-1}\sum_{j=1}^{r_\ell}\lambda H^{n,n-\ell}_{ij}U^{\ell j}\\
	+\begin{cases}\psi_{1i}(0)u_0,&n=1,\\
\sum_{j=1}^{r_{n-1}}K^{n,n-1}_{ij}U^{n-1,j},&2\le n\le N.\end{cases}
\end{multline}
At the $n$th time step, this $r_n\times r_n$ linear system must be solved
to determine $U^{n1}$, $U^{n2}$,\dots, $U^{nr_n}$ and hence $U(t)$ for~$t\in 
I_n$.

\begin{remark}\label{remark: classical limit}
If we send~$\alpha\to1$, so that the fractional \textsc{ODE} 
in~\eqref{eq: scalar ivp} reduces to the classical \textsc{ODE}
$u'+\lambda u=f(t)$, then $H^{n\bar\ell}_{ij}=0$ for $1\le\bar\ell\le n-1$.
Indeed, since $\omega_1(t)=1$, we see that 
$\rho^{n\bar\ell}_j(t)=\int_{I_\ell}\psi_{\ell j}(s)\,ds$ is constant and
so $(\rho^{n\bar\ell}_j)'(t)=0$ for~$t\in I_n$.  Moreover, 
$(\rho^{n0}_j)'(t)=\psi_{nj}(t)$ so 
$H^{n0}_{ij}=\int_{I_n}\psi_{nj}\psi_{ni}\,dt$.
\end{remark}

\begin{remark}\label{remark: dual}
Later we will show certain symmetry properties of~$H^{n0}_{ij}$ using the 
identity
\begin{multline}\label{eq: nn dual}
\int_a^b\biggl(\frac{\partial}{\partial t}\int_a^t
	\omega_\alpha(t-s)u(s)\,ds\biggr)v(t)\,dt\\
=-\int_a^b u(s)\biggl(\frac{\partial}{\partial s}\int_{s}^b
	\omega_\alpha(t-s)v(t)\,dt\biggr)\,ds.
\end{multline}
In fact, the substitution~$x=t-s$ gives
\begin{align*}
\frac{\partial}{\partial t}\int_a^t
	\omega_\alpha(t-s)u(s)\,ds
&=\frac{\partial}{\partial t}\int_0^{t-a}
	\omega_\alpha(x)u(t-x)\,dx\\
&=\omega_\alpha(t-a)u(a)+\int_0^{t-a}
	\omega_\alpha(x)u'(t-x)\,dx\\
&=\omega_\alpha(t-a)u(t_{n-1})+\int_a^t
	\omega_\alpha(t-s)u'(s)\,ds,
\end{align*}
and \eqref{eq: nn dual} follows after reversing the order of 
integration and then integrating by parts.  Similarly, for $\ell\le n-1$,
\begin{equation}\label{eq: nl dual}
\int_a^b\biggl(\frac{\partial}{\partial t}\int_a^b
	\omega_\alpha(t-s)u(s)\,ds\biggr)v(t)\,dt
=-\int_a^bu(s)\biggl(\frac{\partial}{\partial s}\int_a^b
	\omega_\alpha(t-s)v(t)\,dt\biggr)\,ds.
\end{equation}
\end{remark}
\section{Evaluation of the coefficients}\label{sec: evaluation}
To compute $G^n_{ij}$, $H^{n\ell}_{ij}$~and $K^{n,n-1}_{ij}$ it is 
convenient to map each closed subinterval~$\bar I_n=[t_{n-1},t_n]$ to the 
reference element~$[-1,1]$.  We therefore define the affine 
function~$\mathsf{t}_n:[-1,1]\to\bar I_n$ by
\[
\mathsf{t}_n(\tau)=\tfrac12\bigl[(1-\tau)t_{n-1}+(1+\tau)t_n\bigr]
    \quad\text{for $-1\le\tau\le1$,}
\]
and let
\[
\Psi_{nj}(\tau)=\psi_{nj}(t)\quad\text{for $t=\mathsf{t}_n(\tau)$
and $-1\le\tau\le1$.}
\]
In this way,
\begin{equation}\label{eq: Gn}
G^n_{ij}=\Psi_{nj}(-1)\Psi_{ni}(-1)
	+\int_{-1}^1\Psi_{nj}'(\tau)\Psi_{ni}(\tau)\,d\tau
\end{equation}
and
\begin{equation}\label{eq: Kn}
K^{n,n-1}_{ij}=\Psi_{n-1,j}(+1)\Psi_{ni}(-1).
\end{equation}
Both of these coefficients are readily computed; the remainder of this
section is devoted to~$H^{n\bar\ell}_{ij}$.  The formulae in the next lemma 
allow us to compute~$H^{n0}_{ij}$ to machine precision via Gauss--Legendre and 
Gauss--Jacobi quadrature.

\begin{lemma}\label{lem: Hnn}
If we define the polynomial
\[
\Phi^n_{ij}(y)=\frac12\int_{-1}^1
    \Psi_{nj}\bigl(\tfrac12(1-y)(1+z)-1\bigr)
    \Psi'_{ni}\bigl(1-\tfrac12(1-y)(1-z)\bigr)\,dz,
\]
then
\begin{align*}
H^{n0}_{ij}=\frac{(k_n/2)^\alpha}{\Gamma(\alpha)}&\biggl(
    \Psi_{ni}(1)\int_{-1}^1(1-\sigma)^\alpha\Psi_{nj}(\sigma)\,d\sigma\\
&\qquad{}-\int_{-1}^1(1+y)^{\alpha-1}(1-y)\Phi^n_{ij}(y)\,dy\biggr).
\end{align*}
\end{lemma}
\begin{proof}
Since $\rho^{n0}_j(t_{n-1})=0$, integration by parts gives
\begin{align*}
H^{n0}_{ij}&=\rho^{n0}_j(t_n)\psi_{ni}(t_n)
	-\int_{I_n}\rho^{n0}_j(t)\psi'_{ni}(t)\,dt\\
	&=\rho^{n0}_j(t_n)\Psi_{ni}(1)-\int_{-1}^1
	\rho^{n0}_j\bigl(\mathsf{t}_n(\tau)\bigr)\Psi_{ni}'(\tau)\,d\tau,
\end{align*}
and since $\mathsf{t}_n(\tau)-\mathsf{t}_n(\sigma)=(\tau-\sigma)k_n/2$, the
substitution $s=\mathsf{t}_n(\sigma)$ yields
\begin{align*}
\rho^{n0}_j\bigl(\mathsf{t}_n(\tau)\bigr)&=\frac{k_n}{2}\int_{-1}^\tau
	\omega_\alpha\bigl(\mathsf{t}_n(\tau)-\mathsf{t}_n(\sigma)\bigr)
	\Psi_{nj}(\sigma)\,d\sigma\\
	&=\frac{(k_n/2)^\alpha}{\Gamma(\alpha)}
	\int_{-1}^\tau(\tau-\sigma)^{\alpha-1}\Psi_{nj}(\sigma)\,d\sigma.
\end{align*}
Thus, 
\[
H^{n0}_{ij}=\frac{(k_n/2)^\alpha}{\Gamma(\alpha)}\biggl(\Psi_{ni}(1)
	\int_{-1}^1(1-\sigma)^{\alpha-1}\Psi_{nj}(\sigma)\,d\sigma
	-B^n_{ij}\biggr),
\]
where
\[
B^n_{ij}=\int_{-1}^1\int_{-1}^\tau(\tau-\sigma)^{\alpha-1}\Psi_{nj}(\sigma)
	\,d\sigma\,\Psi_{ni}'(\tau)\,d\tau.
\]
We make the substitution $1+y=\tau-\sigma$, which results in a fixed 
singularity at~$y=-1$, and then reverse the order of integration:
\begin{align*}
B^n_{ij}&=\int_{-1}^1\int_{-1}^\tau(1+y)^{\alpha-1}\Psi_{nj}(\tau-y-1)\,dy
	\,\Psi_{ni}'(\tau)\,d\tau\\
	&=\int_{-1}^1(1+y)^{\alpha-1}\int_y^1
	\Psi_{nj}(\tau-y-1)\Psi_{ni}'(\tau)\,d\tau\,dy.
\end{align*}
The substitution $\tau=\tfrac12\bigl[(1-z)y+(1+z)\bigr]$ then yields
\[
\int_y^1\Psi_{nj}(\tau-y-1)\Psi_{ni}'(\tau)\,d\tau\,dy=(1-y)\Phi^n_{ij}(y),
\]
implying the desired formula for~$H^{n0}_{ij}$.
\end{proof}

To deal with~$H^{n,n-\ell}_{ij}$ for~$\ell\le n-1$, we introduce the notation
\[
t_{n-1/2}=\mathsf{t}_n(0)=\tfrac12(t_{n-1}+t_n)\quad\text{and}\quad
D_{n\bar\ell}=D_{n,n-\ell}=t_{n-1/2}-t_{\ell-1/2},
\]
with
\[
\Delta_{n\bar\ell}(\tau,\sigma)=\Delta_{n,n-\ell}(\tau,\sigma)
	=\frac{\tau k_n-\sigma k_\ell}{2D_{n\bar\ell}},
\]
so that
\[
\mathsf{t}_n(\tau)-\mathsf{t}_\ell(\sigma)
=D_{n\bar\ell}\bigl(1+\Delta_{n\bar\ell}(\tau,\sigma)\bigr).
\]

\begin{lemma}\label{lem: H A B C}
If $1\le\ell\le n-1$, then
\[
H^{n\bar\ell}_{ij}
	=\frac{D_{n\bar\ell}^{\alpha-1}}{\Gamma(\alpha)}\,\frac{k_\ell}{2}
	\bigl(\Psi_{ni}(1)\mathcal{A}^{n\bar\ell}_j-
	\Psi_{ni}(-1)\mathcal{B}^{n\bar\ell}_j
	-\mathcal{C}^{n\bar\ell}_{ij}\bigr),
\]
where
\begin{align*}
\mathcal{A}^{n\bar\ell}_j&=\int_{-1}^1
\bigl(1+\Delta_{n\bar\ell}(1,\sigma)\bigr)^{\alpha-1}\Psi_{\ell j}(\sigma)
\,d\sigma,\\
\mathcal{B}^{n\bar\ell}_j&=\int_{-1}^1
\bigl(1+\Delta_{n\bar\ell}(-1,\sigma)\bigr)^{\alpha-1}\Psi_{\ell j}(\sigma)
\,d\sigma,\\
\mathcal{C}^{n\bar\ell}_{ij}&=\int_{-1}^1\Psi_{ni}'(\tau)\int_{-1}^1
	\bigl(1+\Delta_{n\bar\ell}(\tau,\sigma)\bigr)^{\alpha-1}
	\Psi_{\ell j}(\sigma)\,d\sigma\,d\tau.
\end{align*}
\end{lemma}
\begin{proof}
Integrating by parts, we find that
\begin{align*}
H^{n\bar\ell}_{ij}&=\rho^{n\bar\ell}_j(t_n)\psi_{ni}(t_n)
	-\rho^{n\bar\ell}_j(t_{n-1})\psi_{ni}(t_{n-1})
	-\int_{I_n}\rho^{n\bar\ell}_j(t)\psi_{ni}'(t)\,dt\\
	&=\rho^{n\bar\ell}_j(t_n)\Psi_{ni}(1)
	-\rho^{n\bar\ell}_j(t_{n-1})\Psi_{ni}(-1)
	-\int_{-1}^1\rho^{n\bar\ell}_j\bigl(\mathsf{t}_n(\tau)\bigr)
		\Psi'_{ni}(\tau)\,d\tau.
\end{align*}
The substitution~$s=\mathsf{t}_\ell(\sigma)$ gives
\[
\rho^{n\bar\ell}_j\bigl(\mathsf{t}_n(\tau)\bigr)
	=\frac{D_{n\bar\ell}^{\alpha-1}}{\Gamma(\alpha)}\,\frac{k_\ell}{2}
	\int_{-1}^1\bigl(1+\Delta_{n\bar\ell}(\tau,\sigma)\bigr)^{\alpha-1}
	\Psi_{\ell j}(\sigma)\,d\sigma,
\]
and the formula for~$H^{n\bar\ell}_{ij}$ follows at once.
\end{proof}

Notice that
\[
1+\Delta_{n\bar\ell}(1,\sigma)
=\frac{2(t_n-t_\ell)+(1-\sigma)k_\ell}{k_n+2(t_{n-1}-t_l)+k_\ell}>0
\quad\text{for $1\le\ell\le n-1$,}
\]
so the integrand of~$\mathcal{A}^{n\bar\ell}_{ij}$ is always smooth.  However,
\[
1+\Delta_{n\bar\ell}(-1,\sigma)
=\frac{2(t_{n-1}-t_\ell)+(1-\sigma)k_\ell}{k_n+2(t_{n-1}-t_l)+k_\ell},
\]
so the integrands of $\mathcal{B}^{n\bar\ell}_j$~and $\mathcal{C}^{n\bar\ell}_j$
are weakly singular if~$\bar\ell=1$ (i.e., if $\ell=n-1$).  The next lemma 
provides alternative expressions that are amenable to Gauss--Jacobi and 
Gauss--Legendre quadrature.

\begin{lemma}\label{lem: B C n-1}
Let $\rho_n=k_n/k_{n-1}$.  Then,
\begin{align*}
\mathcal{A}^{n1}_j&=(1+\rho_n)^{1-\alpha}
\int_{-1}^1(2\rho_n+1-\sigma)^{\alpha-1}\Psi_{n-1,j}(\sigma)\,d\sigma,\\
\mathcal{B}^{n1}_j&=(1+\rho_n)^{1-\alpha}
\int_{-1}^1(1-\sigma)^{\alpha-1}\Psi_{n-1,j}(\sigma)\,d\sigma
\end{align*}
and
\begin{multline*}
\mathcal{C}^{n1}_{ij}=(1+\rho_n)^{1-\alpha}\biggl(\\
\int_{-1}^1(1+\tau)^\alpha\Psi'_{ni}(\tau)\int_0^1(\rho_n+z)^{\alpha-1}
	\Psi_{n-1,j}\bigl(1-z(1+\tau)\bigr)\,dz\,d\tau\\
+\int_{-1}^1(1-\sigma)^\alpha\Psi_{n-1,j}(\sigma)\int_0^1
(\rho_nz+1)^{\alpha-1}\Psi_{ni}'\bigl(z(1-\sigma)-1\bigr)\,dz\,d\sigma\biggr).
\end{multline*}
\end{lemma}
\begin{proof}
Since $1+\Delta_{n1}(-1,\sigma)=k_{n-1}(1-\sigma)/(k_n+k_{n-1})$, the
formula for~$\mathcal{B}^{n1}_{ij}$ follows at once.  To deal 
with~$\mathcal{C}^{n,1}_{ij}$ we begin by mapping~$[-1,1]^2$ onto~$[0,2]^2$ 
with the substitution~$(\tau,\sigma)=(x-1,1-y)$.  In this way, the singularity
at~$(\tau,\sigma)=(-1,1)$ moves to~$(x,y)=(0,0)$, and
\[
\mathcal{C}^{n1}_{ij}=\int_0^2\int_0^2
	\bigl(1+\Delta_{n1}(x-1,1-y)\bigr)^{\alpha-1}
	\Psi_{n-1,j}(1-y)\Psi_{ni}'(x-1)\,dx\,dy
\]
with
\[
1+\Delta_{n1}(x-1,1-y)=\frac{xk_n+yk_{n-1}}{k_n+k_{n-1}}.
\]
By splitting the integration domain~$[0,2]^2$ into the triangular halves where 
$x>y$~and $x<y$, we obtain
\begin{align*}
\mathcal{C}^{n1}_{ij}&=\int_0^2\Psi_{ni}'(x-1)\int_0^x
\biggl(\frac{xk_n+yk_{n-1}}{k_n+k_{n-1}}\biggr)^{\alpha-1}
\Psi_{n-1,j}(1-y)\,dy\,dx\\
&\qquad{}+\int_0^2\Psi_{n-1,j}(1-y)\int_0^y
\biggl(\frac{xk_n+yk_{n-1}}{k_n+k_{n-1}}\biggr)^{\alpha-1}
\Psi_{ni}'(x-1)\,dx\,dy.
\end{align*}
The substitution~$y=zx$ tranforms the inner integral in the first term to
\[
x^\alpha\int_0^1\biggl(\frac{k_n+zk_{n-1}}{k_n+k_{n-1}}\biggr)^{\alpha-1}
	\Psi_{n-1,j}(1-zx)\,dz,
\]
and the substitution $x=zy$ transforms that in the second to
\[
y^\alpha\int_0^1\biggl(\frac{zk_n+k_{n-1}}{k_n+k_{n-1}}\biggr)^{\alpha-1}
	\Psi_{ni}'(zy-1)\,dz.
\]
Thus,
\begin{align*}
\mathcal{C}^{n1}_{ij}&=\int_0^2x^\alpha\Psi'_{ni}(x-1)\int_0^1
\biggl(\frac{k_n+zk_{n-1}}{k_n+k_{n-1}}\biggr)^{\alpha-1}
	\Psi_{n-1,j}(1-zx)\,dz\,dx\\
&\qquad{}+\int_0^2y^\alpha\Psi_{n-1,j}(1-y)\int_0^1
\biggl(\frac{zk_n+k_{n-1}}{k_n+k_{n-1}}\biggr)^{\alpha-1}
	\Psi_{ni}'(zy-1)\,dz\,dy.
\end{align*}
Now make the substitutions $x=1+\tau$ and $y=1-\sigma$.
\end{proof}

We also have the following alternative representation.

\begin{lemma}\label{lem: Hnl alt n-2}
If $1\le\ell\le n-2$, then
\[
H^{n\bar\ell}_{ij}=-\frac{1-\alpha}{\Gamma(\alpha)}\,\frac{k_nk_\ell}{4}\,
	D_{n\bar\ell}^{\alpha-2}
	\int_{-1}^1\Psi_{ni}(\tau)\int_{-1}^1
	\bigl(1+\Delta_{n\ell}(\tau,\sigma)\bigr)^{\alpha-2}
	\Psi_{\ell j}(\sigma)\,d\sigma\,d\tau.
\]
\end{lemma}
\begin{proof}
When $\ell\le n-2$,
\[
(\rho^{n\bar\ell}_j)'(t)=\int_{I_\ell}\omega_{\alpha-1}(t-s)\psi_{\ell j}(s)\,ds
\quad\text{for $t>t_\ell$,}
\]
and so
\begin{equation}\label{eq: Hnl n-2}
H^{n\bar\ell}_{ij}=\int_{I_n}\psi_{ni}(t)\int_{I_\ell}\omega_{\alpha-1}(t-s)
	\psi_{\ell j}(s)\,ds.
\end{equation}
The result now follows via the substitutions $t=\mathsf{t}_n(\tau)$~and
$s=\mathsf{t}_\ell(\sigma)$, noting that 
$\Gamma(\alpha)=(\alpha-1)\Gamma(\alpha-1)$.
\end{proof}

\begin{remark}\label{remark: uniform}
If the time levels are uniformly spaced, and if the reference basis 
functions are the same for each subinterval, say
\[
k_\ell=k,\quad r_\ell=r\quad\text{and}\quad\Psi_{\ell j}=\Psi_j
\quad\text{for $1\le\ell\le n$ and $1\le j\le r$,}
\]
then
\[
D_{n\bar\ell}=\bar\ell k\quad\text{and}\quad
\Delta_{n\bar\ell}(\tau,\sigma)=\frac{\tau-\sigma}{2\bar\ell},
\]
so the formulae of \cref{lem: H A B C} show that $H^{n\bar\ell}_{ij}$ depends
on $n$~and $\ell$ only through the difference~$\bar\ell=n-\ell$; for
further details, see \cref{example: uniform} below.
\end{remark}
\section{Spatial discretisation}\label{sec: spatial}
The initial-boundary value problem~\eqref{eq: fpde} is known to be 
well-posed~\cite{LeMcLeanStynes2019,McLean2010,McLeanEtAl2019}.  Let
$\iprod{u,v}=\int_\Omega uv$ denote the usual inner product in~$L^2(\Omega)$,
and let $a(u,v)$ denote the bilinear form associated with~$A$ via the first 
Green identity.  For example, if $A=-\nabla^2$ then 
$a(u,v)=\int_\Omega\nabla u\cdot\nabla v$. In this way, the weak 
solution~$u:(0,T]\to H^1_0(\Omega)$ satisfies
\[
\iprod{\partial_t u,v}+a(\partial_t^{1-\alpha}u,v)=\iprod{f(t),v}
    \quad\text{for $v\in H^1_0(\Omega)$ and $0<t\le T$.}
\]
We choose a finite dimensional subspace $V_n\subseteq H^1_0(\Omega)$
for~$0\le n\le N$, and form the vector~$\boldsymbol{V}=(V_1,\ldots,V_N)$. For 
example, $V_n$ might be a (conforming) finite element space constructed using a 
triangulation of~$\Omega$. Our trial 
space~$\mathcal{X}=\mathcal{X}(\boldsymbol{t},\boldsymbol{r},\boldsymbol{V})$
then consists of the functions $X:I\to H^1_0(\Omega)$ such that
$X|_{I_n}\in\mathbb{P}_{r_n-1}(I_n;V_n)$, that is, the restriction~$X|_{I_n}$
is a polynomial in~$t$ of degree at most~$r_n-1$, with coefficients from~$V_n$.
Generalising~\eqref{eq: scalar dG}, the \textsc{dG} solution~$U\in\mathcal{X}$
of~\eqref{eq: fpde} satisfies
\begin{equation}\label{eq: dG}
\bigiprod{\jump{U}^{n-1},X^{n-1}_+}+\int_{I_n}\iprod{\partial_t U,X}\,dt
    +\int_{I_n}a(\partial_t^{1-\alpha}U,X)\,dt=\int_{I_n}\iprod{f(t),X}\,dt
\end{equation}
for $X\in\mathbb{P}_{r_n-1}(I_n;V_n)$ and $1\le n\le N$, with $U^0_-=U_0$ for a 
suitable $U_0\in V_0$ such that $U_0\approx u_0$.

We choose a basis $\{\phi_{np}\}_{p=1}^{P_n}$ for~$V_n$. In the 
expansion~\eqref{eq: U psi}, the coefficient~$U^{nj}$ is now a function 
in~$V_n$, so there exist real numbers $U^{nj}_q$ such that
\[
U^{nj}(x)=\sum_{q=1}^{P_n}U^{nj}_q\phi_{nq}(x)
\quad\text{for $x\in\Omega$;}
\]
for example, $U^{nj}_q=U^{nj}(x_{nq})$ if $x_{nq}$ is the $q$th free node of a 
finite element mesh and if $\phi_{nq}$ is the corresponding nodal basis 
function.  Similarly, for the discrete initial data, there are real 
numbers~$U_{0q}$ such that
\[
U_0(x)=\sum_{q=1}^{P_0}U_{0q}\phi_{0q}(x)\quad\text{for $x\in\Omega$.}
\]

Choosing $X(x,t)=\psi_{ni}(t)\phi_{nq}(x)$ in~\eqref{eq: dG}, we find that
the equations~\eqref{eq: basic linear system} for time stepping the scalar
problem generalise to
\begin{multline}\label{eq: linear equations}
\sum_{j=1}^{r_n}\sum_{q=1}^{P_n}\bigl(G^n_{ij}M^{nn}_{pq}+
H^{n0}_{ij}A^{nn}_{pq}\bigr)U^{nj}_q=F^{ni}_p
	-\sum_{\ell=1}^{n-1}\sum_{j=1}^{r_\ell}\sum_{q=1}^{P_\ell}
H^{n,n-\ell}_{ij}A^{n\ell}_{pq}U^{\ell j}_q\\
	+\begin{cases}\psi_{1i}(0)\sum_{q=1}^{P_0}M^{10}_{pq}U_{0q},
    &n=1,\\[2\jot]
\sum_{j=1}^{r_{n-1}}\sum_{q=1}^{P_{n-1}}K^{n,n-1}_{ij}M^{n,n-1}_{pq}
    U^{n-1,j}_q,&2\le n\le N,\end{cases}
\end{multline}
where
\[
M^{n\ell}_{pq}=\iprod{\phi_{\ell q},\phi_{np}},\qquad
A^{n\ell}_{pq}=a(\phi_{\ell q},\phi_{np}),\qquad
F^{ni}_p=\int_{I_n}\iprod{f(t),\phi_{np}}\psi_{ni}(t)\,dt.
\]
By introducing the $P_n\times P_\ell$ mass matrix 
$\boldsymbol{M}^{n\ell}=[M^{n\ell}_{pq}]$~and 
stiffness matrix~$\boldsymbol{A}^{n\ell}=[A^{n\ell}_{pq}]$, and forming the 
column vectors
\[
\boldsymbol{U}^{nj}=\begin{bmatrix}
U^{nj}_1\\ U^{nj}_2\\ \vdots\\ U^{nj}_{P_n}\end{bmatrix},\qquad
\boldsymbol{F}^{ni}=\begin{bmatrix}
F^{ni}_1\\ F^{ni}_2\\ \vdots\\ F^{ni}_{P_n}\end{bmatrix},\qquad
\boldsymbol{U}_0=\begin{bmatrix}
U_{01}\\ U_{02}\\ \vdots\\ U_{0P_0}\end{bmatrix},
\]
we can rewrite the equations~\eqref{eq: linear equations} as
\begin{multline}\label{eq: linear equations vector}
\sum_{j=1}^{r_n}\bigl(G^n_{ij}\boldsymbol{M}^{nn}+
H^{n0}_{ij}\boldsymbol{A}^{nn}\bigr)\boldsymbol{U}^{nj}
	=\boldsymbol{F}^{ni}
	-\sum_{\ell=1}^{n-1}\sum_{j=1}^{r_\ell}
H^{n,n-\ell}_{ij}\boldsymbol{A}^{n\ell}\boldsymbol{U}^{\ell j}\\
	+\begin{cases}\psi_{1i}(0)\boldsymbol{M}^{10}\boldsymbol{U}_0,
    &n=1,\\[2\jot]
\sum_{j=1}^{r_{n-1}}K^{n,n-1}_{ij}\boldsymbol{M}^{n,n-1}\boldsymbol{U}^{n-1,j},
    &2\le n\le N.\end{cases}
\end{multline}

To write \eqref{eq: linear equations vector} even more compactly, define
the $r_n\times r_n$ matrix~$\boldsymbol{G}^n=[G^n_{ij}]$ and the 
$r_n\times r_\ell$ matrix~$\boldsymbol{H}^{n\bar\ell}=[H^{n\bar\ell}_{ij}]$, 
together with the (block) column vectors
\[
\boldsymbol{U}^n=\begin{bmatrix}
\boldsymbol{U}^{n1}\\
\boldsymbol{U}^{n2}\\
\vdots\\
\boldsymbol{U}^{nr_n}\end{bmatrix}
\quad\text{and}\quad
\boldsymbol{F}^n=\begin{bmatrix}
\boldsymbol{F}^{n1}\\
\boldsymbol{F}^{n2}\\
\vdots\\
\boldsymbol{F}^{nr_n}\end{bmatrix}.
\]
We also form the $r_n\times r_{n-1}$ 
matrix~$\boldsymbol{K}^{n,n-1}=[K^{n,n-1}_{ij}]$ and the column vector
\[
\boldsymbol{\psi^0_+}=\begin{bmatrix}
\psi_{11}(0)\\                          
\psi_{12}(0)\\                          
\vdots\\
\psi_{1r_n}(0)\\                          
\end{bmatrix}.
\]
Utilising the Kronecker product, the linear 
system~\eqref{eq: linear equations vector} takes the form
\begin{multline}\label{eq: kronecker}
\bigl(\boldsymbol{G}^n\otimes\boldsymbol{M}^{nn}+
\boldsymbol{H}^{n0}\otimes\boldsymbol{A}^{nn}\bigr)\boldsymbol{U}^n
=\boldsymbol{F}^n-\sum_{\ell=1}^{n-1}
\bigl(\boldsymbol{H}^{n,n-\ell}\otimes\boldsymbol{A}^{n\ell}\bigr)
    \boldsymbol{U}^\ell\\
	+\begin{cases}
\bigl(\boldsymbol{\psi}^0_+\otimes\boldsymbol{M}^{10}\bigr)
    \boldsymbol{U}_0,&n=1,\\[2\jot]
\bigl(\boldsymbol{K}^{n,n-1}\otimes\boldsymbol{M}^{n,n-1}\bigr)
\boldsymbol{U}^{n-1,j},&2\le n\le N.\end{cases}
\end{multline}
\section{Legendre polynomials}\label{sec: Legendre}
Let $P_0$, $P_1$, $P_2$, \dots denote the Legendre polynomials with the
standard normalisation $P_j(1)=1$ for all~$j\ge0$.  By choosing
\begin{equation}\label{eq: Psi P}
\Psi_{nj}(\tau)=P_{j-1}(\tau),
\end{equation}
we obtain a convenient and well-conditioned basis for~$\mathbb{P}_{r_n-1}$ 
with the properties
\[
\int_{-1}^1\Psi_{nj}(\tau)\Psi_{ni}(\tau)\,d\tau=\frac{2\delta_{ij}}{2j-1}
\quad\text{and}\quad\Psi_{nj}(-\tau)=(-1)^{j-1}\Psi_{nj}(\tau)
\]
for $i$, $j\in\{1,2,\ldots,r_n\}$.

\begin{lemma}
With the choice~\eqref{eq: Psi P} of basis functions, 
\begin{equation}\label{eq: Psi pm 1}
\Psi_{nj}(1)=1\quad\text{and}\quad
\Psi_{nj}(-1)=(-1)^{j-1},
\end{equation}
and the coefficients \eqref{eq: Gn}~and \eqref{eq: Kn} are given by
\[
G^n_{ij}=\begin{cases}
(-1)^{i+j},&\text{if $i\ge j$,}\\
1,&\text{if $i<j$,}\end{cases}
\]
and
\[
K^{n,n-1}_{ij}=(-1)^{i-1}.
\]
\end{lemma}
\begin{proof}
The properties~\eqref{eq: Psi pm 1} follow from $P_j(1)=1$~and 
$P_j(-1)=(-1)^j$.  Hence, the formula for~$K^{n,n-1}_{ij}$ follows
from~\eqref{eq: Kn}, and by~\eqref{eq: Gn},
\[
G^n_{ij}=(-1)^{i+j}+E_{ij}\quad\text{where}\quad
E_{ij}=\int_{-1}^1P_{j-1}'(\tau)P_{i-1}(\tau)\,d\tau.
\]
If $j\le i$, then $E_{ij}=0$ because $P_{j-1}'$ is orthogonal to~$P_{i-1}$.
Otherwise, if $j>i$, then $P_{j-1}$ is orthogonal to~$P_{i-1}'$ so
integration by parts gives
\[
E_{ij}=\bigl[P_{j-1}(x)P_{i-1}(x)\bigr]_{-1}^1
    -\int_{-1}^1P_{j-1}(x)P_{i-1}'(x)\,dx=1-(-1)^{i+j}
\]
and hence $G^n_{ij}=1$.
\end{proof}

\begin{example}
If $r_n=4$ and $r_{n-1}=3$, then the matrices 
$\boldsymbol{G}^n=[G^n_{ij}]$~and $\boldsymbol{K}^{n,n-1}=[K^{n,n-1}_{ij}]$
are
\[
\boldsymbol{G}^n=\left[\begin{array}{rrrr}
 1& 1& 1& 1\\
-1& 1& 1& 1\\
 1&-1& 1& 1\\
-1& 1&-1&\phantom{-}1\end{array}\right]\quad\text{and}\quad
\boldsymbol{K}^{n,n-1}=\left[\begin{array}{rrr}
 1& 1& 1\\
-1&-1&-1\\
 1& 1& 1\\
-1&-1&-1\end{array}\right].
\]
\end{example}

We have no analogous, simple formula for the remaining 
coefficients~$H^{n\ell}_{ij}$.  However, when $\bar\ell=0$ ($\ell=n$) the 
following parity property holds.

\begin{lemma}\label{lem: Hn0 parity}
With the choice~\eqref{eq: Psi P} of basis functions, 
\begin{equation}\label{eq: Hn0 parity}
H^{n0}_{ji}=(-1)^{i+j}H^{n0}_{ij}.
\end{equation}
\end{lemma}
\begin{proof}
Using \eqref{eq: nn dual}, we find that
\[
H^{n0}_{ji}=\int_{-1}^1(BP_{i-1})'(\tau)P_{j-1}(\tau)\,d\tau
    =-\int_{-1}^1P_{i-1}(\sigma)(B^*P_{j-1})'(\sigma)\,d\sigma,
\]
where
\[
(Bv)(\tau)=\int_{-1}^\tau\omega_\alpha(\tau-\sigma)v(\sigma)\,d\sigma
\quad\text{and}\quad
(B^*v)(\sigma)=\int_\sigma^1\omega_\alpha(\tau-\sigma)v(\tau)\,d\tau.
\]
Let $(RV)(\tau)=V(-\tau)$.  A short calculation shows that $RB^*=BR$, so
\begin{align*}
(B^*P_{j-1})'(-\sigma)&=-\frac{d}{d\sigma}\bigl[(B^*P_{j-1})(-\sigma)\bigr]
=-\frac{d}{d\sigma}(RB^*P_{j-1})'(\sigma)\\
&=-(BRP_{j-1})'(\sigma)=(-1)^j(BP_{j-1})'(\sigma),
\end{align*}
and therefore, using the substitution~$\sigma=-x$,
\begin{align*}
H^{n0}_{ji}&=(-1)^{j+1}\int_{-1}^1P_{i-1}(-x)(BP_{j-1})'(x)\,dx\\
    &=(-1)^{i+j}\int_{-1}^1(BP_{j-1})'(x)P_{i-1}(x)\,dx=(-1)^{i+j}H^{n0}_{ij},
\end{align*}
as claimed.
\end{proof}

\begin{remark}\label{remark: H alpha=1}
In the limit as~$\alpha\to1$, we see from \cref{remark: classical limit} that
\[
H^{n0}_{ij}\to\int_{I_n}\psi_{nj}(t)\psi_{ni}(t)\,dt
	=\frac{k_n}{2}\int_{-1}^1\Psi_j(\tau)\Psi_i(\tau)\,d\tau
	=\frac{k_n\delta_{ij}}{2j-1}.
\]
\end{remark}

\begin{example}\label{example: uniform}
Consider the uniform case $k_n=k$, $r_n=r$ and $\Psi_{nj}=\Psi_j$ 
for $1\le n\le N$ (as in \cref{remark: uniform}), with
$\Psi_j(\tau)=P_{j-1}(\tau)$ as above.  We then have
\[
H^{n\bar\ell}_{ij}=k^\alpha H^{\bar\ell}_{ij}
    \quad\text{for $1\le\ell\le n\le N$ and $i$, $j\in\{1,2,\ldots,r\}$,}
\]
where, by \cref{lem: Hnn},
\begin{equation}\label{eq: H<0>}
\begin{aligned}
H^0_{ij}=\frac{1}{2^\alpha\Gamma(\alpha)}&\biggl(
    \int_{-1}^1(1-\sigma)^\alpha P_{j-1}(\sigma)\,d\sigma\\
&\qquad{}-\int_{-1}^1(1+y)^{\alpha-1}(1-y)\Phi_{ij}(y)\,dy\biggr),
\end{aligned}
\end{equation}
with
\begin{equation}\label{eq: Phi ij l=0}
\Phi_{ij}(y)=\frac12\int_{-1}^1
    P_{j-1}\bigl(\tfrac12(1-y)(1+z)-1\bigr)
    P'_{i-1}\bigl(1-\tfrac12(1-y)(1-z)\bigr)\,dz,
\end{equation}
and by \cref{lem: H A B C},
\[
H^{\bar\ell}_{ij}=\frac{\bar\ell^{\alpha-1}}{2\Gamma(\alpha)}\bigl(
\mathcal{A}^{\bar\ell}_j+(-1)^i\mathcal{B}^{\bar\ell}_j
-\mathcal{C}^{\bar\ell}_{ij}\bigr)\quad\text{for $\ell\ge1$,}
\]
with, letting $\Delta_{\bar\ell}(\tau)=\tau/(2\bar\ell)$,
\begin{align*}
\mathcal{A}^{\bar\ell}_j&=\int_{-1}^1
	\bigl(1+\Delta_{\bar\ell}(1-\sigma)\bigr)^{\alpha-1}
	P_{j-1}(\sigma)\,d\sigma,\\
\mathcal{B}^{\bar\ell}_j&=\int_{-1}^1
	\bigl(1-\Delta_{\bar\ell}(1+\sigma)\bigr)^{\alpha-1}
	P_{j-1}(\sigma)\,d\sigma,\\
\mathcal{C}^{\bar\ell}_{ij}&=\int_{-1}^1 P_{i-1}'(\tau)\int_{-1}^1
	\bigl(1+\Delta_{\bar\ell}(\tau-\sigma)\bigr)^{\alpha-1}
	P_{j-1}(\sigma)\,d\sigma\,d\tau.
\end{align*}
Moreover, \cref{lem: B C n-1} provides alternative expressions 
when~$\bar\ell=1$:
\begin{align*}
\mathcal{A}^1_{ij}&=2^{1-\alpha}\int_{-1}^1
	(3-\sigma)^{\alpha-1}P_{j-1}(\sigma)\,d\sigma,\\
\mathcal{B}^1_{ij}&=2^{1-\alpha}\int_{-1}^1
	(1-\sigma)^{\alpha-1}P_{j-1}(\sigma)\,d\sigma
\end{align*}
and
\begin{align*}
\mathcal{C}^1_{ij}&=2^{1-\alpha}\biggl(
	\int_{-1}^1(1+\tau)^\alpha P_{i-1}'(\tau)
	\int_0^1(1+z)^{\alpha-1}P_{j-1}\bigl(1-z(1+\tau)\bigr)\,dz\,d\tau\\
&{}+\int_{-1}^1(1-\sigma)^\alpha P_{j-1}(\sigma)
	\int_0^1(z+1)^{\alpha-1}P_{i-1}'\bigl(z(1-\sigma)-1\bigr)\,dz\,d\sigma
\biggr).
\end{align*}
Likewise, \cref{lem: Hnl alt n-2} provides an alternative expression 
for~$\bar\ell\ge2$:
\begin{equation}\label{eq: H<l> alt}
H^{\bar\ell}_{ij}=-\frac{1-\alpha}{4\Gamma(\alpha)}\,\bar\ell^{\alpha-2}
	\int_{-1}^1P_{i-1}(\tau)\int_{-1}^1
	\bigl(1+\Delta_{\bar\ell}(\tau-\sigma)\bigr)^{\alpha-2}
	P_{j-1}(\sigma)\,d\sigma\,d\tau.
\end{equation}
Finally, by arguing as in the proof of \cref{lem: Hn0 parity}, we can show that
\begin{equation}\label{eq: H<l> parity}
H^{\bar\ell}_{ji}=(-1)^{i+j}H^{\bar\ell}_{ij}\quad\text{for all $\bar\ell\ge0$.}
\end{equation}
\end{example}
\section{Reconstruction}\label{sec: reconstruction}
Throughout this section, we continue to use the Legendre 
basis~\eqref{eq: Psi P}. Some insight into the \textsc{dG} method can be had by 
considering the trivial case of~\eqref{eq: fpde} when~$A=0$, that is, 
$\partial_tu=f(t)$ for~$0<t\le T$, with $u(0)=u_0$.  The \textsc{dG} 
scheme~\eqref{eq: dG} then reduces to
\begin{equation}\label{eq: dG A=0}
\bigiprod{\jump{U}^n,X^n_+}+\int_{I_n}\iprod{\partial_t U,X}\,dt
    =\int_{I_n}\iprod{\partial_tu,X}\,dt
\end{equation}
for $X\in\mathbb{P}_{r_n-1}(I_n;V_n)$ and $1\le n\le N$, with~$U^0_-=U_0$.
To state our next result, let $\mathcal{P}_n$ denote the orthoprojector 
from~$L_2(\Omega)$ onto~$V_n$, and define
\[
\mathcal{Q}_{n\ell}=\mathcal{P}_n\mathcal{P}_{n-1}\cdots\mathcal{P}_{\ell+1}.
\]

\begin{lemma}
If $A=0$ and $U_0=\mathcal{P}_0u_0$, then for~$1\le n\le N$ the 
\textsc{dG} solution~$U\in\mathcal{X}$ satisfies
\begin{equation}\label{eq: Un-}
U^n_-=\mathcal{P}_nu(t_n)+\sum_{\ell=0}^{n-1}
    \mathcal{Q}_{n\ell}(\mathcal{P}_\ell-I)u(t_\ell)
\end{equation}
and
\begin{equation}\label{eq: U-u orthog}
\int_{I_n}\iprod{U-u,\partial_tX}\,dt=0
\quad\text{for all $X\in\mathbb{P}_n(I_n;V)$.}
\end{equation}
\end{lemma}
\begin{proof}
Integrating by parts in~\eqref{eq: dG A=0}, we find that
\[
\iprod{U^n_--u(t_n),X^n_-}=\iprod{U^{n-1}_--u(t_{n-1}),X^n_+}
    +\int_{I_n}\iprod{U-u,\partial_tX}\,dt.
\]
Given $v\in V_n$, by choosing the constant function~$X(t)=v$ for~$t\in I_n$
we deduce that $\iprod{U^n_--u(t_n),v}=\iprod{U^{n-1}_--u(t_{n-1}),v}$ and
so \eqref{eq: U-u orthog} is satisfied.  Moreover,
\[
\mathcal{P}_n\bigl(U^n_--u(t_n)\bigr) 
    =\mathcal{P}_n\bigl(U^{n-1}_--u(t_{n-1})\bigr),
\]
and, by the choice of initial condition, we see that
\eqref{eq: Un-} is satisfied for~$n=1$:
\begin{align*}
U^1_--\mathcal{P}_1u(t_1)&=\mathcal{P}_1\bigl(U^1_--u(t_1)\bigr)
    =\mathcal{P}_1(I-\mathcal{P}_0+\mathcal{P}_0)\bigl(U^0_--u(t_0)\bigr)\\
    &=\mathcal{P}_1(\mathcal{P}_0-I)u(t_0)+\mathcal{P}_1(U_0-\mathcal{P}_0u_0)
    =\mathcal{Q}_{11}(\mathcal{P}_0-I)u(t_0).
\end{align*}
Letting $n\ge2$, we make the induction hypothesis 
\[
U^{n-1}_-=\mathcal{P}_{n-1}u(t_{n-1})+\sum_{\ell=0}^{n-1}
    \mathcal{Q}_{n-1,\ell}(\mathcal{P}_\ell-I)u(t_\ell),
\]
and observe that
\begin{align*}
U^n_--\mathcal{P}_nu(t_n)&=\mathcal{P}_n\bigl(U^n_--u(t_n)\bigr)
    =\mathcal{P}_n(I-\mathcal{P}_{n-1}+\mathcal{P}_{n-1})
        \bigl(U^{n-1}_--u(t_{n-1})\bigr)\\
    &=\mathcal{P}_n(\mathcal{P}_{n-1}-I)u(t_{n-1})
    +\mathcal{P}_n\sum_{\ell=0}^{n-1}\mathcal{Q}_{n-1,\ell}
        (\mathcal{P}_\ell-I)u(t_\ell),
\end{align*}
which gives the desired formula~\eqref{eq: Un-}.
\end{proof}

For the remainder of this section, we will assume that the subspaces~$V_n$
are nested, as follows:
\begin{equation}\label{eq: Vn nesting}
V_0\supseteq V_1\supseteq V_2\supseteq\cdots\supseteq V_N.
\end{equation}
It follows that $\mathcal{P}_{\ell+1}(\mathcal{P}_\ell-I)=0$ 
for~$0\le\ell\le N-1$ and so 
\begin{equation}\label{eq: Un- simple}
U^n_-=\mathcal{P}_nu(t_n).  
\end{equation}
The following explicit representation for~$U$ holds.

\begin{lemma}\label{lem: U expansion}
If $A=0$, $U_0=\mathcal{P}_0u_0$ and the subspaces 
satisfy~\eqref{eq: Vn nesting}, then
\begin{equation}\label{eq: U expansion}
U(t)=\sum_{j=1}^{r_n-1}a_{nj}\psi_{nj}(t)+\tilde a_n\psi_{nr_n}(t)
    \quad\text{for $t\in I_n$,}
\end{equation}
where
\[
a_{nj}=\frac{2j-1}{k_n}\int_{I_n}\mathcal{P}_nu(t)\psi_{nj}(t)\,dt
\]
are the local Fourier--Legendre coefficients of~$\mathcal{P}_nu$, but
\[
\tilde a_n=\mathcal{P}_nu(t_n)-\sum_{j=1}^{r_n-1}a_{nj}.
\]
\end{lemma}
\begin{proof}
By definition, $U|_{I_n}\in\mathbb{P}_{r_n-1}(I_n;V_n)$ so there exist
coefficients $a_{nj}$~and $\tilde a_n$ in~$V_n$ such that $U$ has the desired
expansion.  The formula for~$a_{nj}$ follows at once from the orthogonality
property of the~$\psi_{nj}$ (see \cref{remark: H alpha=1}).  The formula
for~$\tilde a_n$ follows from~\eqref{eq: Un- simple} because 
$\psi_{nj}(t_n)=P_{j-1}(1)=1$ for all~$j$.
\end{proof}

We have a Peano kernel~$\mathsf{G}_r$ for the Fourier--Legendre expansion of 
degree~$r$,
\[
f(\tau)=\sum_{j=1}^{r+1}b_j\Psi_j(\tau)
    +\int_{-1}^1\mathsf{G}_r(\tau,\sigma)f^{(r+1)}(\sigma)\,d\sigma
    \quad\text{for $-1\le\tau\le1$,}
\]
assuming $f:[-1,1]\to\mathbb{R}$ is $C^{r+1}$, and also a Peano 
kernel~$\mathsf{M}_j(\tau)$ for the $j$th coefficient:
\[
b_j=\frac{2j-1}{2}\int_{-1}^1f(\tau)\Psi_j(\tau)\,d\tau
=\int_{-1}^1\mathsf{M}_j(\tau)f^{(j-1)}(\tau)\,d\tau.
\]
Thus, if $t=\mathsf{t}_n(\tau)$~and $s=\mathsf{t}_n(\sigma)$, and if we 
define the local Peano kernels
\[
\mathsf{g}_{nr}(t,s)=(k_n/2)^r\mathsf{G}_r(\tau,\sigma)
\quad\text{and}\quad
\mathsf{m}_{nj}(t)=(k_n/2)^{j-2}\mathsf{M}_j(\tau),
\]
then
\begin{equation}\label{eq: Pu expansion}
\mathcal{P}_nu(t)=\sum_{j=1}^{r_n+1}a_{nj}\psi_{nj}(t)
    +\int_{I_n}\mathsf{g}_{r_n}(t,s)\mathcal{P}_nu^{(r_n+1)}(s)\,ds
\quad\text{for $t\in I_n$,}
\end{equation}
and
\[
a_{nj}=\int_{I_n}\mathsf{m}_{nj}(s)\mathcal{P}_nu^{(j-1)}(s)\,ds.
\]
It follows that $a_{nj}=O(k_n^{j-1})$ provided $u$ is~$C^{j-1}$ on~$\bar I_n$.

\begin{theorem}\label{thm: Pu-U}
Assume that $A=0$, $U_0=\mathcal{P}u_0$ and the subspaces
satisfy~\eqref{eq: Vn nesting}.  If $u:\bar I_n\to L^2(\Omega)$ is $C^{r_n+1}$, 
then $a_{n,r_n+1}=O(k_n^{r_n})$ and
\begin{equation}\label{eq: Pu-U}
\mathcal{P}_nu(t)-U(t)=a_{n,r_n+1}\bigl[\psi_{n,r_n+1}(t)-\psi_{n,r_n}(t)\bigr]
    +O(k_n^{r_n+1})\quad\text{for $t\in I_n$.}
\end{equation}
\end{theorem}
\begin{proof}
Subtracting \eqref{eq: U expansion} from~\eqref{eq: Pu expansion}, we have
\[
\mathcal{P}_nu(t)-U(t)=(a_{n,r_n}-\tilde a_n)\psi_{nr_n}(t)
+a_{n,r_n+1}\psi_{n,r_n+1}(t)+O(k_n^{r_n+1})
\]
for $t\in I_n$.  Since $U^n_-=\mathcal{P}_nu(t_n)$~and
$\psi_{n,r_n}(t_n)=\psi_{n,r_n+1}(t_n)=1$, taking the limit as~$t\to t_n$
yields $a_{n,r_n}-\tilde a_n=-a_{n,r_n+1}+O(k_n^{r_n+1})$.
\end{proof}

\begin{corollary}\label{cor: jump U}
$\mathcal{P}_n\jump{U}^{n-1}=2(-1)^{r_n+1}a_{n,r_n+1}+O(k_n^{r_n+1})$.
\end{corollary}
\begin{proof}
As $t\to t_{n-1}^+$, the left-hand side of~\eqref{eq: Pu-U} tends to
\begin{align*}
\mathcal{P}_nu(t_{n-1})-U^{n-1}_+&=
\mathcal{P}_n(I-\mathcal{P}_{n-1}+\mathcal{P}_{n-1})U^{n-1}_--U^{n-1}_+
=\mathcal{P}_nU^{n-1}_--U^{n-1}_+\\
&=-\mathcal{P}_n(U^{n-1}_+-U^{n-1}_-)=-\mathcal{P}_n\jump{U}^{n-1}, 
\end{align*}
and on the right-hand side, $\psi_{n,r_n+1}(t)-\psi_{n,r_n}(t)$ tends 
to~$P_{r_n}(-1)-P_{r_n-1}(-1) =(-1)^{r_n}-(-1)^{r_n-1}=2(-1)^{r_n}$.
\end{proof}

\begin{figure}
\caption{The polynomials $P_r(\tau)-P_{r-1}(\tau)$.}
\label{fig: Radau polys}
\begin{center}
\includegraphics[scale=0.75]{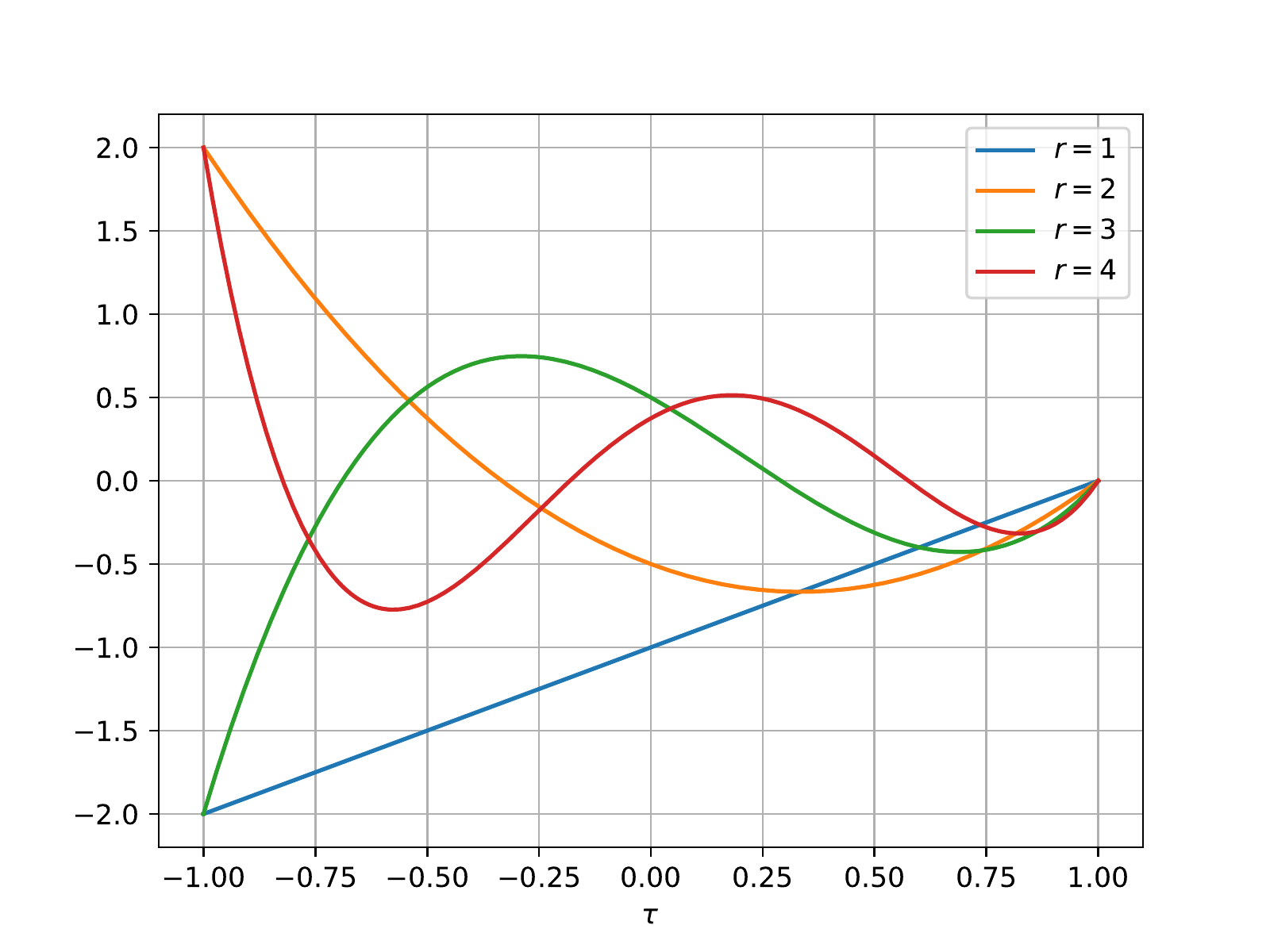}
\end{center}
\end{figure}

In light of \cref{thm: Pu-U}, we consider the polynomials
\[
\psi_{n,r_n+1}(t)-\psi_{n,r_n}(t)=\Psi_{r_n+1}(\tau)-\Psi_{r_n}(\tau)
    =P_{r_n}(\tau)-P_{r_n-1}(\tau).
\]
As illustrated in \cref{fig: Radau polys}, there are $r+1$~points
\[
-1=\tau_{r0}<\tau_{r1}<\cdots<\tau_{rr}=1
\]
such that
\[
(P_r-P_{r-1})(\tau_{rj})=0\quad\text{for $1\le j\le r$.}
\]
In fact, the $r$~zeros $\tau_{r1}$, $\tau_{r2}$, \dots, $\tau_{rr}$
are the points of a right-Radau quadrature
rule~\cite[Chapter~9]{Krylov1962} on the interval~$[-1,1]$.  We put
\begin{equation}\label{eq: t*nj}
t^*_{nj}=\mathsf{t}_n(\tau_{r_nj})\quad\text{for $0\le j\le r_n$,}
\end{equation}
so that $t^*_{n-1}=t^*_{n0}<t_{n1}<\cdots<t^*_{nr_n}=t_n$ and
\[
\psi_{n,r_n+1}(t^*_{nj})-\psi_{n,r_n}(t^*_{nj})=0
\quad\text{for $1\le j\le r_n$.}
\]

From \cref{thm: Pu-U}, we see that $\mathcal{P}_nu(t)-U(t)=O(k_n^{r_n})$
for general~$t\in I_n$, but
$\mathcal{P}_nu(t^*_{nj})-U(t^*_{nj})=O(k_n^{r_n+1})$ for~$1\le j\le r_n$.
Let $\widehat{\mathcal{X}}$ denote the space obtained from~$\mathcal{X}$ by
increasing the maximum allowed polynomial degree over the subinterval~$I_n$
from $r_n$ to~$\hat r_n=r_n+1$, for~$1\le n\le N$. The 
\emph{reconstruction}~$\widehat U\in\widehat{\mathcal{X}}$ of~$U\in\mathcal{X}$ 
is then defined by requiring that
\[
\widehat U(t^*_{nj})=U(t^*_{nj})\quad\text{for $1\le j\le r_n-1$,}
\]
and that the one-sided limits at the end points are
\[
\widehat U^{n-1}_+=\mathcal{P}_nU^{n-1}_-\quad\text{and}\quad 
\widehat U^n_-=U^n_-.
\]
Since $\widehat U|_{I_n}$ is a polynomial of degree at most~$\hat r_n-1=r_n$, 
it is uniquely determined by these $r_n+1$ interpolation conditions.  Notice
also that $\widehat U$ is continuous at~$t_{n-1}$ if~$V_{n-1}=V_n$ because
$\mathcal{P}_nU^{n-1}_-=U^{n-1}_-$.

Makridakis and Nochetto~\cite{MakridakisNochetto2006} introduced the
reconstruction in their analysis of \emph{a posteriori} error bounds
for parabolic PDEs.  Since the polynomial~$(U-\widehat U)|_{I_n}$ has 
degree at most~$r_n$ and vanishes at~$t^*_{nj}$ for~$1\le n\le r_n$, it must
be a multiple of $\psi_{n,r_n+1}-\psi_{nr_n}$.  In fact, by taking
limits as~$t\to t_{n-1}^+$, we see that
\begin{equation}\label{eq: U-hat U}
U(t)-\widehat U(t)=\frac12(-1)^{r_n}\mathcal{P}_n\jump{U}^{n-1}
    \bigl[\psi_{n,r_n+1}(t)-\psi_{n,r_n}(t)\bigr]
\quad\text{for $t\in I_n$.}
\end{equation}
At the same time, by \cref{thm: Pu-U}~and \cref{cor: jump U},
\begin{equation}\label{eq: U-Pu}
U(t)-\mathcal{P}_nu(t)=\frac12(-1)^{r_n}\mathcal{P}_n\jump{U}^{n-1}
    \bigl[\psi_{n,r_n+1}(t)-\psi_{n,r_n}(t)\bigr]+O(k_n^{r_n+1})
\quad\text{for $t\in I_n$,}
\end{equation}
implying that $\widehat U-\mathcal{P}_nu$ is $O(k_n^{r_n+1})$ on~$I_n$.
One of our principal aims in the next section is to investigate 
numerically the error in the \textsc{dG} solution~$U$ and its
reconstruction~$\widehat U$ in non-trival cases where~$A\ne0$.  We can hope
that something like \eqref{eq: U-Pu} still holds, because the time
derivative in the term~$\partial_t^{1-\alpha}Au$ is of lower order
than in~$\partial_tu$.  Notice that \eqref{eq: U psi}~and
\eqref{eq: U-hat U} imply
\[
\widehat U(t)=\sum_{j=1}^{\hat r_n}\widehat U^{nj}\psi_{nj}(t)
    \quad\text{for $t\in I_n$,}
\]
where
\[
\widehat U^{nj}=\begin{cases}
U^{nj},&1\le j\le r_n-1,\\
U^{nr_n}+\tfrac12(-1)^{r_n}\mathcal{P}_n\jump{U}^{n-1},&j=r_n,\\
\tfrac12(-1)^{r_n+1}\mathcal{P}_n\jump{U}^{n-1},&j=r_n+1=\hat r_n.
\end{cases}
\]
\section{Numerical experiments}\label{sec: numerical}

A Julia package~\cite{McLean2020} provides functions to evaluate
the coefficients $G^n_{ij}$, $K^{n,n-1}_{ij}$ and $H^{n\bar\ell}_{ij}$
based on the results of \cref{sec: evaluation,sec: Legendre}.  This package
also includes (in the \texttt{examples} directory) the scripts used for
the examples below.

\subsection{The matrix $\protect{\boldsymbol{H}^{\bar\ell}}$}
Let $\alpha=3/4$, and consider for simplicity the case when $k_n=k$ and $r_n=r$
are constant for all~$n$, so that the formulae of \cref{example: uniform}
apply.  To get a sense of how the matrix entries~$H^{\bar\ell}_{ij}$ behave, 
we computed
\[
\boldsymbol{H}^{0}=\left[\begin{array}{rrrr}
 1.08807 &  0.15544 &  0.07065 &  0.04239 \\
-0.15544 &  0.49458 &  0.09326 &  0.04834 \\
 0.07065 & -0.09326 &  0.33839 &  0.06893 \\
-0.04239 &  0.04834 & -0.06893 &  0.26319 
\end{array}\right]
\]
and 
\[
\boldsymbol{H}^{1}=\left[\begin{array}{rrrr}
-0.34623 & -0.13428 & -0.06884 & -0.04219\\
 0.13428 &  0.08414 &  0.05405 &  0.03690\\
-0.06884 & -0.05405 & -0.04050 & -0.03048\\
 0.04219 &  0.03690 &  0.03048 &  0.02472
\end{array}\right],
\]
which illustrate the property~\eqref{eq: H<l> parity}.  
The factor~$\bigl(1+\Delta_{\bar\ell}(\tau-\sigma)\bigr)^{\alpha-2}$ 
in~\eqref{eq: H<l> alt} becomes very smooth as~$\bar\ell$ increases,
with the result that $H^{\bar\ell}_{ij}$ decays rapidly to zero as~$i+j$ 
increases.  Even for~$\bar\ell=2$, we have
\[
\boldsymbol{H}^2=10^{-1}\times\left[\begin{array}{rrrr}
-0.91483 & -0.10220 & -0.01261 & -0.00164\\
 0.10220 &  0.02027 &  0.00355 &  0.00059\\
-0.01261 & -0.00355 & -0.00080 & -0.00016\\
 0.00164 &  0.00059 &  0.00016 &  0.00004
\end{array}\right],
\]
and \cref{fig: antidiags} shows this behaviour for larger values of~$\bar\ell$,
with entries in the lower right corner of the matrix reaching the order of the 
machine epsilon ($2^{-52}\approx2.22\times10^{-16}$) once $\bar\ell$ is of 
order~$100$.

\begin{figure}
\caption{Decay of $\max_{i+j=m}|H^{\bar\ell}_{ij}|$ for increasing $m$~and 
$\bar\ell$, when $\alpha=3/4$.}\label{fig: antidiags}
\begin{center}
\includegraphics[scale=0.75]{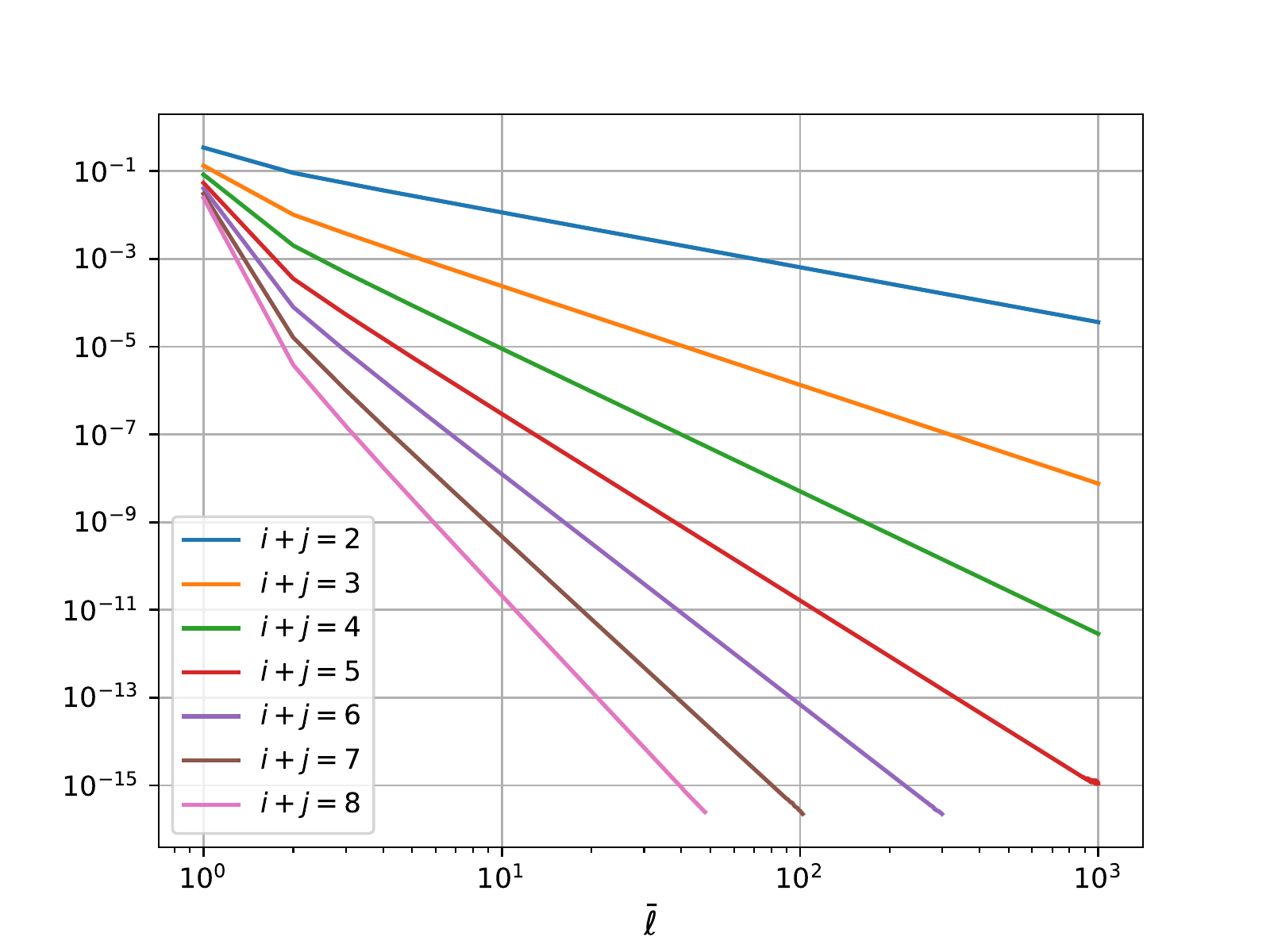}
\end{center}
\end{figure}

The value of~$H^0_{ij}$ can be computed to machine precision using Gauss 
quadrature with $M_\sigma=\lceil j/2\rceil$ and $M_y=\lceil (i+j)/2\rceil-1$ 
points for the integrals with respect to $\sigma$~and $y$ in~\eqref{eq: H<0>}, 
and using $M_z=\lceil(i+j)/2\rceil-1$ points for the integral with respect
to~$z$ in~\eqref{eq: Phi ij l=0}.  When~$\ell\ge1$, let $H^{\bar\ell}_{ij}(M)$ 
denote the value of~$H^{\bar\ell}_{ij}$ computed by applying $M$-point Gauss 
rules to \eqref{eq: H<l> alt}, that is, $M^2$~points for the double integral. 
For a given absolute tolerance~$\mathtt{atol}$, let 
$M^{\bar\ell}_r(\mathtt{atol})$ denote the smallest~$M$ for which 
\[
\bigl|H^{\bar\ell}_{ij}(M)-H^{\bar\ell}_{ij}(12)\bigr|<\mathtt{atol}
\quad\text{for all $i$, $j\in\{1,2,\ldots,r\}$.}
\]
\cref{table: Gauss points} lists some values 
of~$M^{\bar\ell}_r(\mathtt{atol})$ for~$\mathtt{atol}=10^{-14}$.  
Unsurprisingly, fewer quadrature points are needed as~$\bar\ell$ increases.

\begin{table}
\caption{Numbers of Gauss points~$M^{\bar\ell}_r(\mathtt{atol})$ required for 
$\texttt{atol}=10^{-14}$, when $\alpha=3/4$.}
\label{table: Gauss points}
\begin{center}
\begin{tabular}{r|rrrrr}
$r$&$\bar\ell=1$&$\bar\ell=2$&$\bar\ell=10$&$\bar\ell=100$&$\bar\ell=1000$\\
\hline
  1&      9&      9&      5&      3&      2\\
  2&      9&      9&      5&      3&      2\\
  3&      9&      9&      5&      4&      3\\
  4&     10&     10&      6&      4&      3\\
  5&     10&     10&      6&      5&      4\\
  6&     11&     11&      7&      5&      4
\end{tabular}
\end{center}
\end{table}

\begin{figure}
\caption{The exact solution~$u$ of \eqref{eq: scalar ivp} in the 
case~\eqref{eq: scalar ivp params}, together with the piecewise-quadratic
($r=3$) \textsc{dG} solution with $N=3$ subintervals.}
\label{fig: scalar ivp soln}
\begin{center}
\includegraphics[scale=0.75]{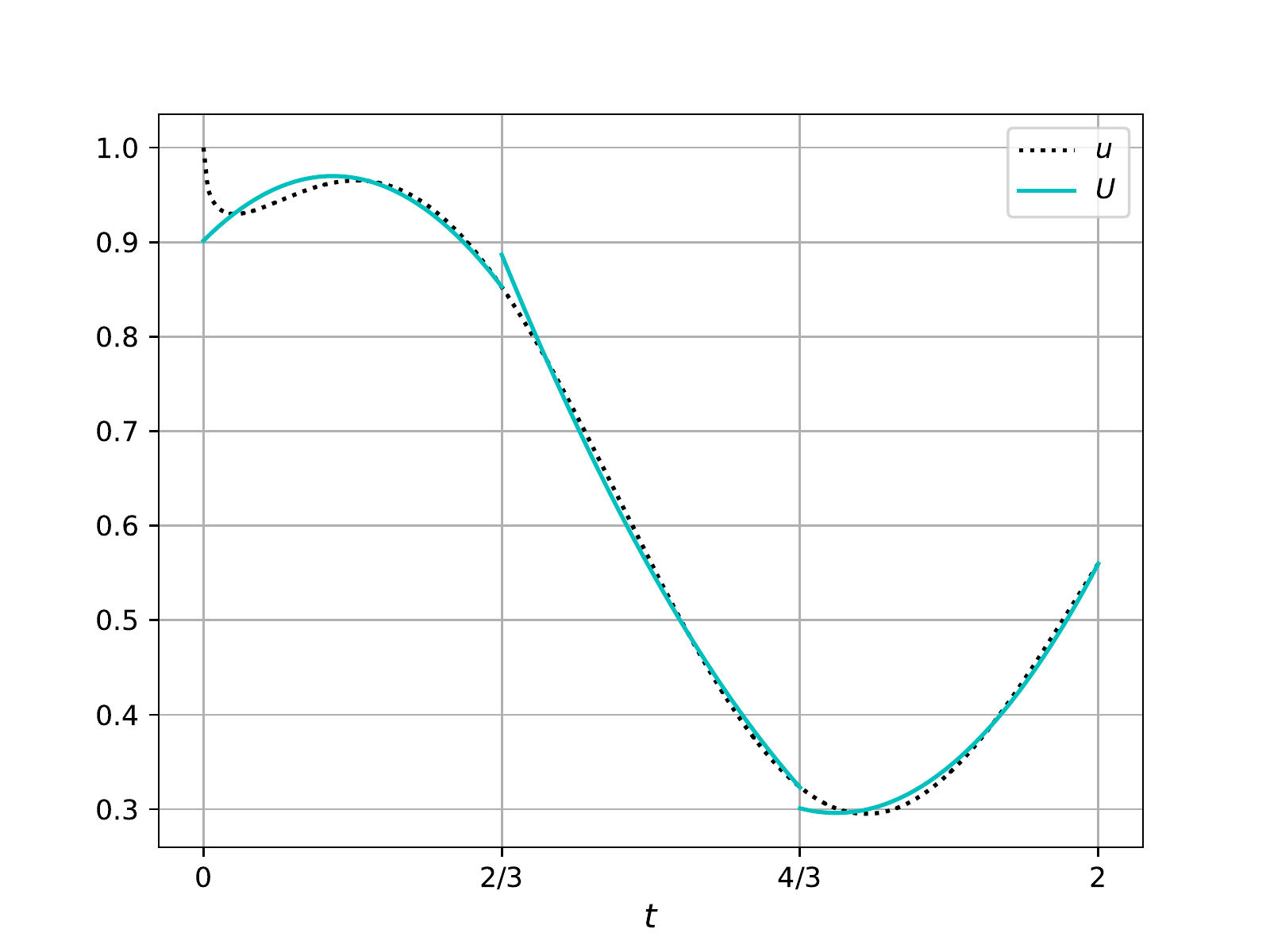}
\end{center}
\end{figure}

\subsection{A fractional \textsc{ODE}}
We consider the initial-value problem~\eqref{eq: scalar ivp} in the case
\begin{equation}\label{eq: scalar ivp params}
\alpha=1/2,\quad\lambda=1/2,\quad f(t)=\cos\pi t,\quad u_0=1,\quad T=2,
\end{equation}
for which the solution is
\[
u(t)=u_0E_{1/2}(-\lambda\sqrt{t})+\int_0^tE_{1/2}(-\lambda\sqrt{t-s})f(s)\,ds,
\]
where $E_\alpha(x)=\sum_{n=0}^\infty t^n/\Gamma(1+n\alpha)$ denotes the
Mittag--Leffler function.  The substitution $s=(1-y^2)t$ yields a smooth 
integrand, allowing $u(t)$ to be computed accurately via Gauss quadrature
on the unit interval~$[0,1]$. Note that $E_{1/2}(-x)=
\operatorname{erfcx}(x)=e^{x^2}\operatorname{erfc(x)}$ is just the scaled 
complementary error function.

\cref{fig: scalar ivp soln} shows $u$, together with the \textsc{dG} 
solution~$U$ using piecewise quadratics ($r=3$) and only $N=3$ subintervals.
In \cref{fig: scalar ivp Enj} we plot the absolute errors,
\begin{equation}\label{eq: Enj}
\widehat E(t)=|\widehat U(t)-u(t)|
\quad\text{and}\quad
E^n_j=\begin{cases}
	|U(t^*_{n0}+0)-u(t^*_{n0})|,&j=0,\\
	|U(t^*_{nj})-u(t^*_{nj})|,&1\le j\le r-1,\\
	|U(t^*_{nr}-0)-u(t^*_{nr})|,&j=r,
\end{cases}
\end{equation}
again using piecewise quadratics but now with $N=5$~subintervals of uniform 
size~$k_n=k=T/N$.  Two features are immediately apparent.  First, the 
accuracy is poor near~$t=0$, reflecting the singular behaviour of the 
solution: for $m\ge1$, the $m$th derivative~$u^{(m)}(t)$ blows up 
like $t^{-(m-1/2)}$ as~$t\to0$.  Second, on intervals~$I_n$ away from~$0$,
the error is notably smaller at the right-Radau points 
($t^*_{nj}$ for $1\le j\le3$) than at the left endpoint ($t^*_{n0}=t_{n-1}$).

\begin{figure}
\caption{Absolute errors in the reconstruction~$\widehat U(t)$ 
for~$0\le t\le T=2$, and in the \textsc{dG} solution~$U(t)$ for~$t=t^*_{nj}$, 
using piecewise quadratics ($r=3$) and $N=5$ uniform subintervals; 
see~\eqref{eq: t*nj}.}
\label{fig: scalar ivp Enj}
\begin{center}
\includegraphics[scale=0.75]{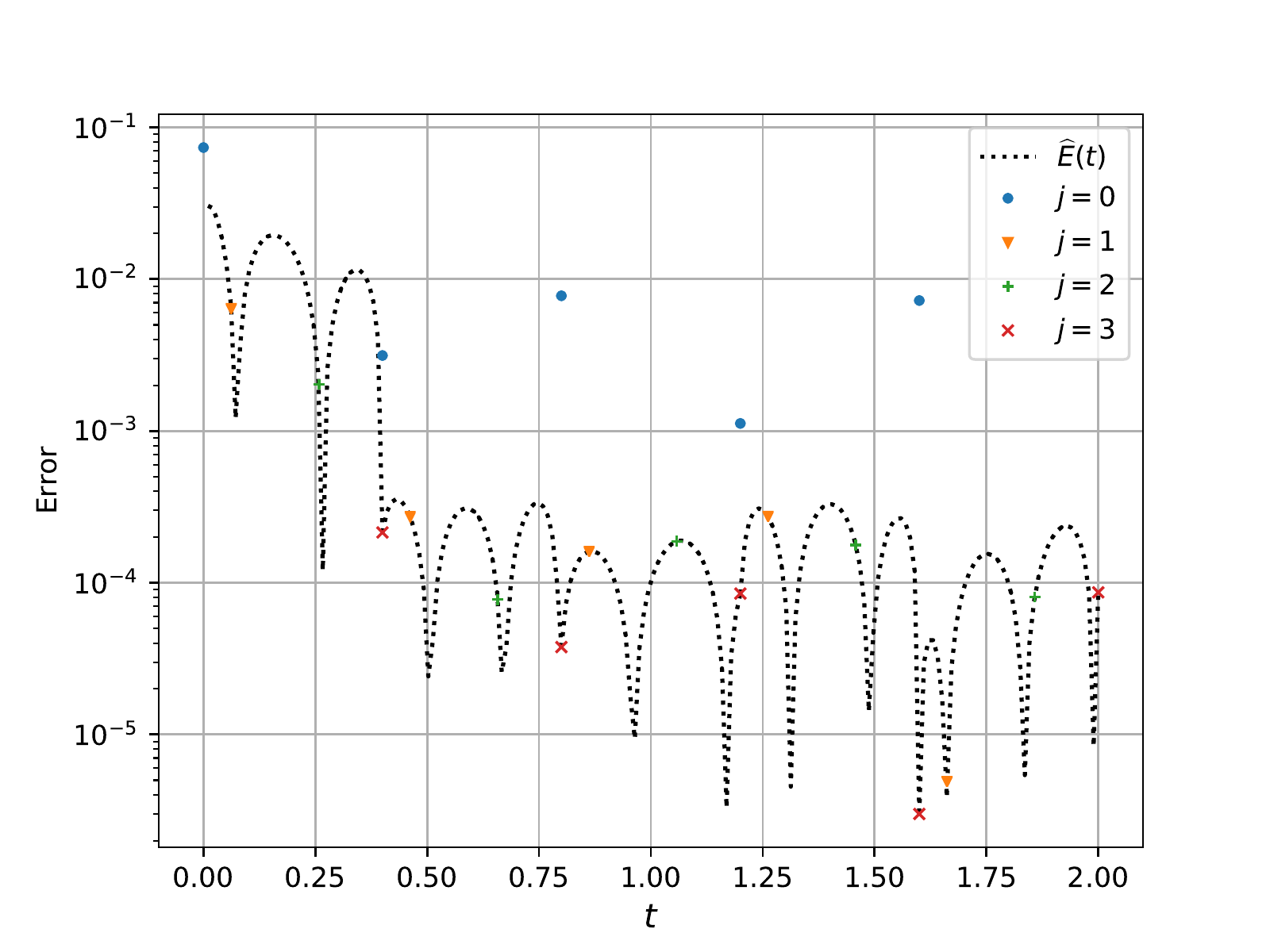}
\end{center}
\end{figure}

\begin{table}
\caption{Maximum weighted errors~\eqref{eq: Ej} at the points~$t^*_{nj}$
using piecewise quadratics ($r=3$) on a uniform grid.}\label{table: Ej r=3}
\renewcommand{\arraystretch}{1.2}
\begin{center}
\texttt{
\begin{tabular}{r|rr|rr|rr|rr|}
\multicolumn{1}{c|}{$N$}&\multicolumn{2}{c|}{$E^{\max}_0$}&
\multicolumn{2}{c|}{$E^{\max}_1$}&\multicolumn{2}{c|}{$E^{\max}_2$}&
\multicolumn{2}{c|}{$E^{\max}_3$}\\
\hline
    8 & 8.0e-03&      & 8.8e-05&      & 1.3e-04&      & 1.0e-04&      \\
   16 & 1.2e-03& 2.69 & 1.4e-05& 2.62 & 1.4e-05& 3.15 & 9.3e-06& 3.46 \\
   32 & 1.7e-04& 2.87 & 1.5e-06& 3.25 & 1.3e-06& 3.40 & 8.2e-07& 3.50 \\
   64 & 2.2e-05& 2.94 & 1.4e-07& 3.42 & 1.2e-07& 3.47 & 7.2e-08& 3.51 \\
  128 & 2.8e-06& 2.97 & 1.3e-08& 3.47 & 1.1e-08& 3.49 & 6.3e-09& 3.51 \\
  256 & 3.6e-07& 2.98 & 1.1e-09& 3.49 & 9.5e-10& 3.50 & 5.5e-10& 3.51 
\end{tabular}
}
\end{center}
\end{table}

In \cref{table: Ej r=3}, we show how the quantities
\begin{equation}\label{eq: Ej}
E^{\max}_j=\max_{1\le n\le N}(t^*_{nj})^{r-\alpha}E^n_j
\end{equation}
behave as~$N$ grows.  These results, together with similar computations
using other choices of $\alpha$~and $r\ge2$, lead us to conjecture that,
in general, using a constant time step~$k$,
\[
E^n_0\le C(t^*_{n0})^{\alpha-r}k^r\quad\text{for $2\le n\le N$,}
\]
whereas
\[
E^n_j\le C(t^*_{nj})^{\alpha-r}k^{r+\alpha}\quad
\text{for $1\le n\le N$ and $1\le j\le r$,}
\]
and that, consequently, 
\[
|\widehat U(t)-u(t)|\le Ct^{\alpha-r}k^{r+\alpha}\quad
\text{for~$t_1\le t\le T$.}
\]
However, using piecewise-constants ($r=1$) we do not observe any
superconvergence, with both $E^{\max}_0$ and $E^{\max}_1$ behaving like
$Ct_n^{1-\alpha}k$, albeit with a noticably smaller constant in the case
of~$E^{\max}_1$.

To suppress the growth in the error as~$t$ approaches $0$, we can use 
a graded mesh of the form
\begin{equation}\label{eq: tn graded}
t_n=(n/N)^qT\quad\text{for $0\le n\le N$,}
\end{equation}
with a suitable grading exponent~$q\ge1$.  \cref{table: recon error q}
shows the maximum error in the reconstruction, i.e.,
$\max_{0\le t\le T}|\widehat U(t)-u(t)|$, together with the associated 
convergence rates, for four choices of~$q$ and using $T=1$ as the final
time.  These errors appear to be of order~$k^{\min(3.5,q\alpha)}$
where $k=\max_{1\le n\le N}k_n\le CN^{-1}$.  We conjecture that, in general,
\begin{equation}\label{eq: Uhat graded error}
|\widehat U(t)-u(t)|\le Ck^{\min(r+\alpha,q\alpha)}
\quad\text{for $0\le t\le T$, provided $r\ge2$.}
\end{equation}

\begin{table}
\caption{Maximum error in the reconstruction~$\widehat U(t)$ 
for~$0\le t\le T=1$, using piecewise quadratics ($r=3$) for four choices of the 
mesh grading exponent~$q$; 
see~\eqref{eq: tn graded}.}\label{table: recon error q}
\renewcommand{\arraystretch}{1.2}
\begin{center}
\texttt{
\begin{tabular}{r|rr|rr|rr|rr|}
\multicolumn{1}{c|}{$N$}&\multicolumn{2}{c|}{$q=1$}&
\multicolumn{2}{c|}{$q=3$}&\multicolumn{2}{c|}{$q=5$}&
\multicolumn{2}{c|}{$q=6$}\\
\hline
   8 & 1.1e-02&      & 1.4e-03&      & 4.1e-04&      & 7.1e-04&      \\
   16 & 6.0e-03& 0.84 & 3.8e-04& 1.89 & 5.2e-05& 2.95 & 9.1e-05& 2.97 \\
   32 & 4.3e-03& 0.50 & 1.3e-04& 1.50 & 7.9e-06& 2.72 & 1.0e-05& 3.19 \\
   64 & 3.0e-03& 0.50 & 4.7e-05& 1.50 & 1.4e-06& 2.51 & 9.8e-07& 3.35 \\
  128 & 2.1e-03& 0.50 & 1.7e-05& 1.50 & 2.5e-07& 2.50 & 9.2e-08& 3.41 \\
  256 & 1.5e-03& 0.50 & 5.9e-06& 1.50 & 4.3e-08& 2.50 & 8.4e-09& 3.45
\end{tabular}
}
\end{center}
\end{table}
\subsection{A fractional PDE}
Consider the elliptic operator~$A=-\partial^2/\partial x^2$ for
the 1D spatial domain~$\Omega=(0,L)$.  To construct a reference solution, 
we exploit fact that the Laplace transform of~$u$,
\[
\tilde u(x,z)=\int_0^\infty e^{-zt}u(x,t)\,dt,
\]
satisfies the two-point boundary-value problem
\[
\omega^2\tilde u-\tilde u_{xx}=g(x,z)
\quad\text{for $0<x<L$,}
\quad\text{with $\tilde u(0,z)=0=\tilde u(L,z)$,}
\]
where
\[
\omega=z^{\alpha/2}\quad\text{and}\quad
g(x,z)=z^{\alpha-1}[u_0(x)+\tilde f(x,z)].
\]
The variation-of-parameters formula leads to the integral 
representation
\begin{multline*}
\tilde u(x,z)=\frac{\sinh\omega(L-x)}{\omega\sinh\omega L}
    \int_0^xg(x,z)\sinh\omega\xi\,d\xi\\
    +\frac{\sinh\omega x}{\omega\sinh\omega L}
    \int_x^Lg(x,z)\sinh\omega\xi\,d\xi,
\end{multline*}
and the Laplace inversion formula then gives
\begin{equation}\label{eq: inv LT}
u(x,t)=\frac{1}{2\pi i}\int_\Gamma e^{zt}\tilde u(x,z)\,dz,
\end{equation}
for a contour~$\Gamma$ homotopic to the imaginary axis 
and passing to the right of all singularities of the 
integrand.

\begin{figure}
\caption{The reference solution for the 1D problem with data given by
\eqref{eq: u0 f}~and \eqref{eq: constants}.}\label{fig: refsoln}
\begin{center}
\includegraphics[scale=0.8]{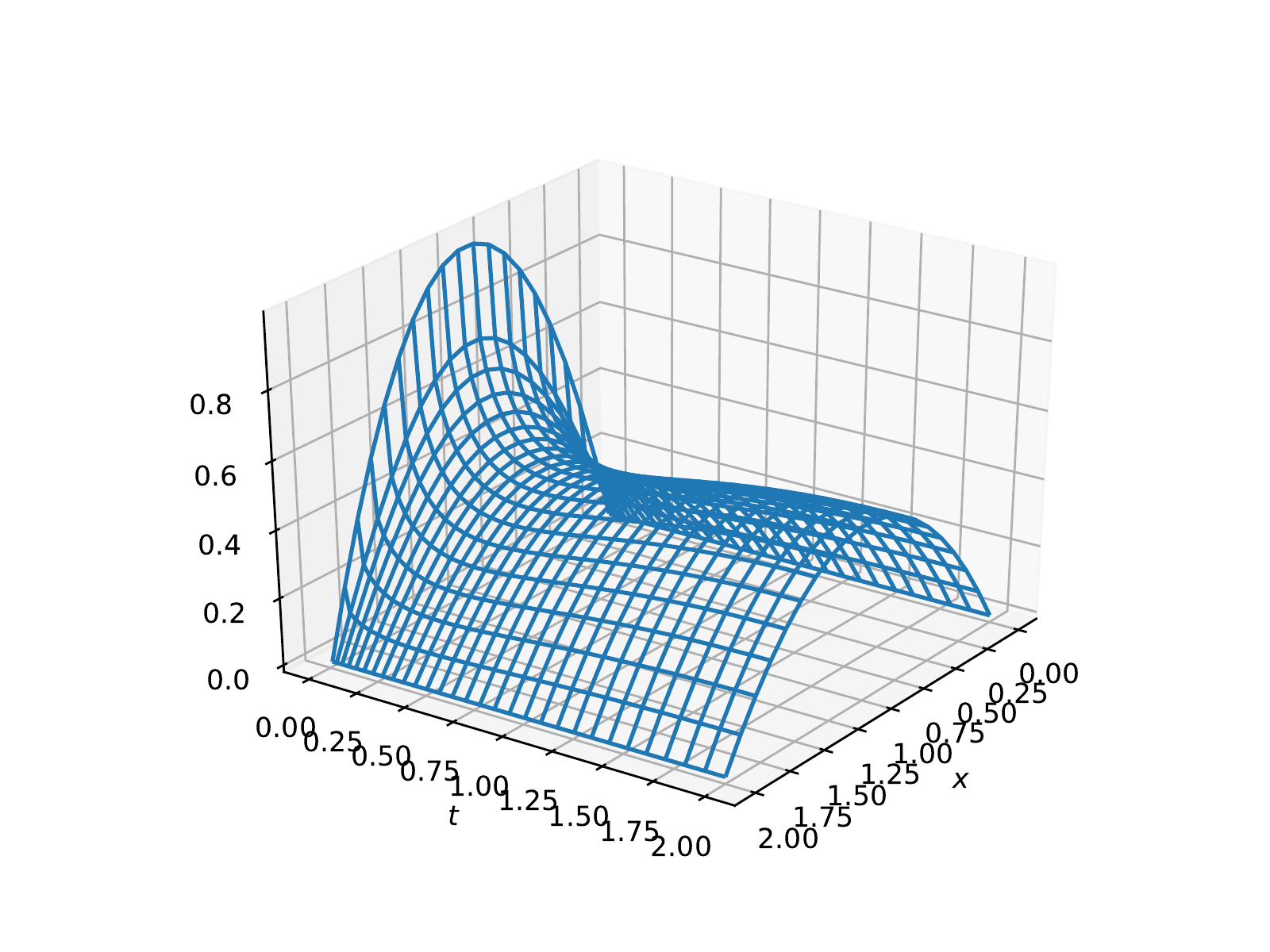}
\end{center}
\end{figure}

\begin{figure}
\caption{Comparison of the jumps~$\|\jump{U}^{n-1}\|$ with the \textsc{dG}
error $\|U(t)-u(t)\|$ and the reconstruction error $\|\widehat U(t)-u(t)\|$.
Top: a uniform mesh with~$N=12$ time steps. Bottom: a graded mesh 
with~$N=40$ time steps.}
\label{fig: jumps FPDE}
\begin{center}
\includegraphics[scale=0.8]{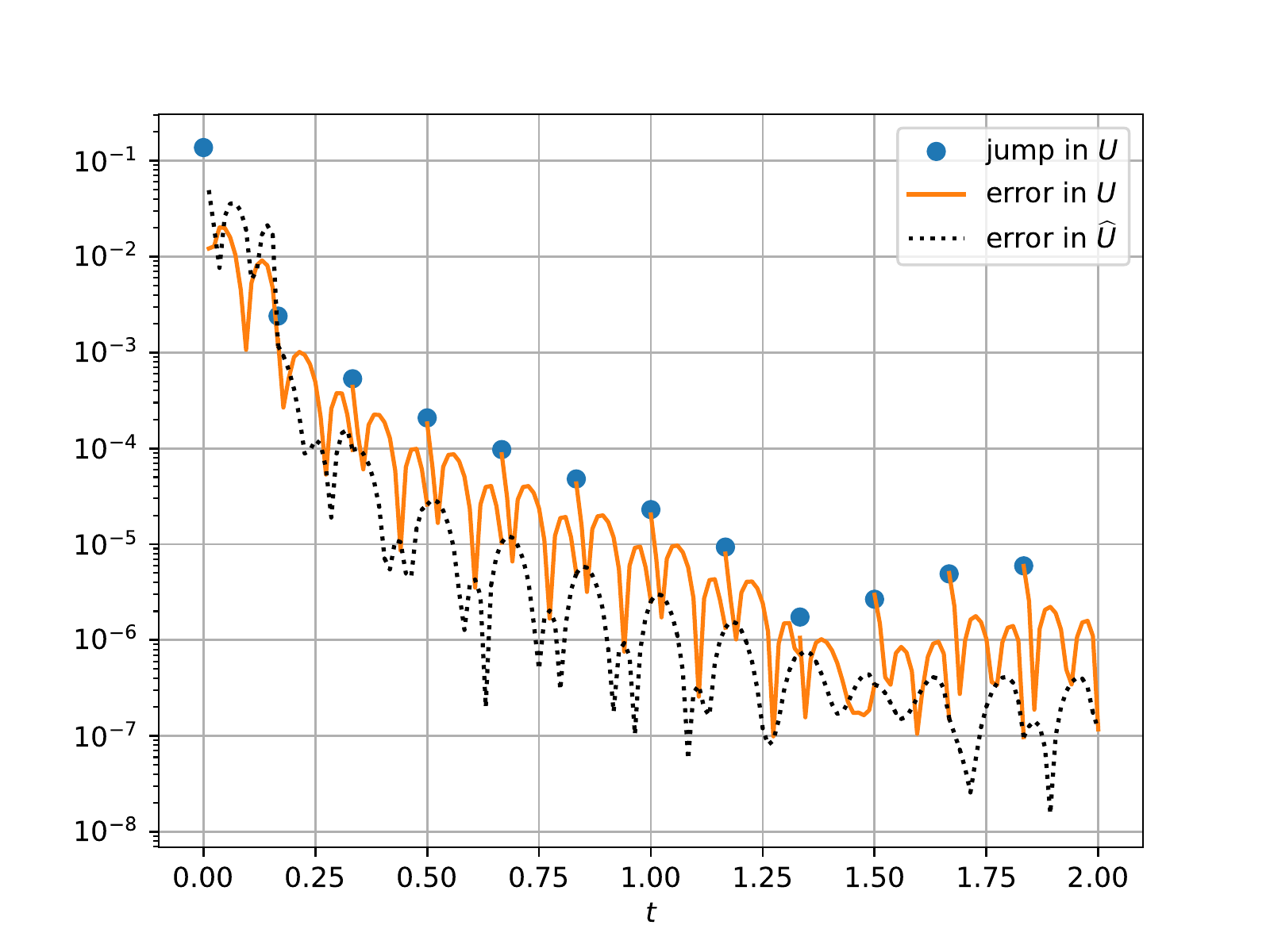}
\includegraphics[scale=0.8]{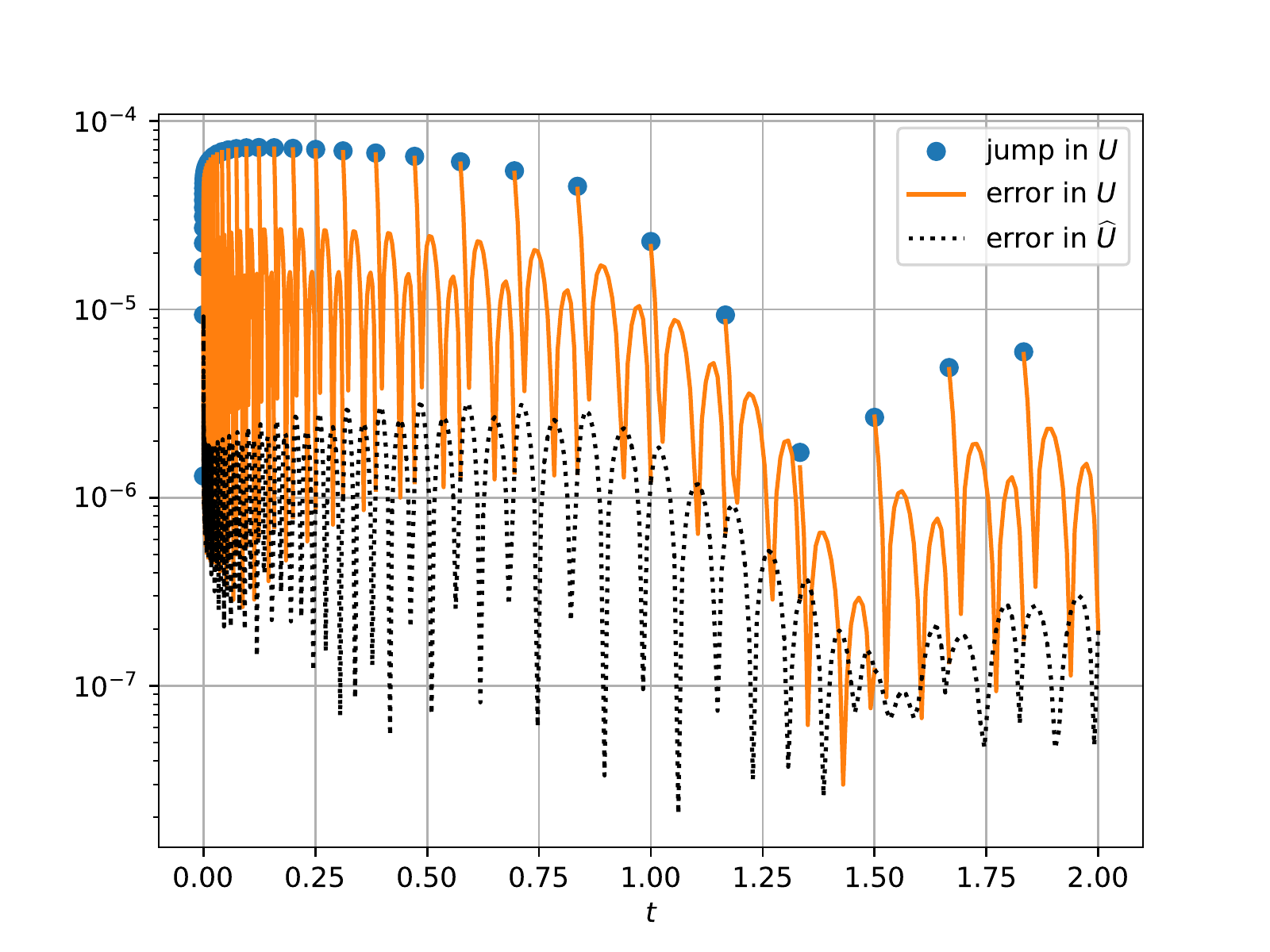}
\end{center}
\end{figure}

We choose as data the functions
\begin{equation}\label{eq: u0 f}
u_0(x)=C_0x(L-x)\quad\text{and}\quad
f(x,t)=C_fte^{-t},
\end{equation}
for constants $C_0$~and $C_f$, and find that
\begin{align*}
\tilde u(x,z)&=\frac{C_0}{z}\,
\frac{\rho_1(x)\sinh\omega(L-x)+\rho_1(L-x)\sinh\omega x}%
{\sinh\omega L}\\
    &\qquad{}+\frac{C_f}{z(z+1)^2}
\frac{\rho_2(x)\sinh\omega(L-x)+\rho_2(L-x)\sinh\omega x}%
{\sinh\omega L},
\end{align*}
where
\[
\rho_1(x)=\bigl(\omega x(L-x)-2\omega^{-1}\bigr)\cosh\omega x
    +(2x-L)\sinh\omega x+2\omega^{-1}
\]
and
\[
\rho_2(x)=\cosh\omega x-1.
\]
To evaluate the contour integral~\eqref{eq: inv LT} we 
apply an optimised equal-weight quadrature rule that arises after deforming
$\Gamma$ into the left branch of an 
hyperbola~\cite{WeidemanTrefethen2007}.  \cref{fig: refsoln} shows
the reference solution over the time interval~$[0,2]$ in the case
\begin{equation}\label{eq: constants}
\alpha=0.6,\qquad L=2,\qquad C_0=1,\qquad C_f=2.
\end{equation}

In~\cref{fig: jumps FPDE}, we plot the $L_2$-norms of the jumps, 
$\|\jump{U}^{n-1}\|$, together with the errors in~$U(t)$ and its 
reconstruction~$\widehat U(t)$.  The \textsc{dG} method used 
piecewise-quadratics ($r=3$), first with a uniform mesh of $N=12$ 
subintervals (top), and then with a non-uniform mesh of $N=40$ subintervals 
(bottom).  In both cases, the spatial discretisation used (continuous) 
piecewise cubics on a uniform grid with $20$~subintervals.  Since $u_0$
is a quadratic polynomial in this instance, we simply put $U_0=u_0$.  
Consistent with our conjecture~\eqref{eq: U-Pu}, we observe that
\[
\sup_{t_{n-1}<t<t_n}\|U(t)-u(t)\|\approx\bigl\|\jump{U}^{n-1}\bigr\|.
\]
Motivated by our conjecture~\eqref{eq: Uhat graded error}, the second mesh
was graded for~$0\le t_n\le 1$ by taking $q=(r+\alpha)/\alpha$, $N=34$~and 
$T=1$ in the formula~\eqref{eq: tn graded}, followed by a uniform mesh on
the other half~$[1,2]$ of the time interval.  We see that the mesh grading
is effective at resolving the solution for~$t$ near zero, albeit with a
substantial increase in the overal computational cost.
\paragraph{Acknowledgements}
This project was supported by a UNSW Faculty Research Grant
(PS47152/IR001/MATH).
\printbibliography
\end{document}